\documentclass{amsart}
\usepackage{eucal}
\usepackage{amsmath}
\usepackage{amsthm,amssymb,times,stmaryrd,color,enumerate,accents,mathtools}
\usepackage{ marvosym }

\usepackage[usenames,dvipsnames]{xcolor}
\usepackage[shortlabels]{enumitem}
\usepackage{url}
\usepackage{inconsolata} 
\definecolor{sangria}{rgb}{0.57, 0.0, 0.04}
\definecolor{royalblue}{rgb}{0.0, 0.22, 0.66}
\usepackage[colorlinks=true, urlcolor=sangria ,hyperindex, linkcolor=royalblue, pagebackref=false, citecolor=sangria]{hyperref}
\usepackage{physics} 
\usepackage{bm}
\usepackage[bb=ams, cal=cm, scr=rsfso]{mathalfa}
\usepackage{graphicx}
\usepackage{tikz-cd}  
\tikzset{
  symbol/.style={
    draw=none,
    every to/.append style={
      edge node={node [sloped, allow upside down, auto=false]{$#1$}}}
  }
}
\usetikzlibrary{calc}

\usepackage{marginnote}
\usepackage[textcolor=blue,
    linecolor=blue!10!white,
    bordercolor=blue!10!white,
    backgroundcolor=white, ]{todonotes}
\usepackage{fullpage}
\usepackage[all]{xy}
\theoremstyle{plain}

\newtheorem{thm}[equation]{Theorem}
\newtheorem{prop}[equation]{Proposition}
\newtheorem{lem}[equation]{Lemma}

\newtheorem{cor}[equation]{Corollary}

\newtheorem{conj}[equation]{Conjecture}
\theoremstyle{definition}
\newtheorem{rmk}[equation]{Remark}

\newtheorem{example}[equation]{Example}
\newtheorem{defn}[equation]{Definition}

\newtheorem{setup}[equation]{Setup}

\numberwithin{equation}{section}
\numberwithin{figure}{section}

\usepackage{xfrac}

\newcommand{\CB}[1]{\left\{#1\right\}}

\newcommand{\DP}[1]{(\!(#1)\!)}
\newcommand{\DB}[1]{\llbracket#1\rrbracket}
\newcommand{\RG}[1]{\langle#1\rangle}
\newcommand{\pma}[1]{{\begin{pmatrix}#1\end{pmatrix}}}

\renewcommand{\mod}{\bmod}
\newcommand{\ra}{\rightarrow}
\newcommand{\xra}[1]{\xrightarrow{#1}}

\newcommand{\mono}{\hookrightarrow}
\newcommand{\epi}{\twoheadrightarrow}

\newcommand{\risom}{\buildrel\sim\over\rightarrow} 
\newcommand{\ov}{\overline}

\newcommand{\til}{\widetilde}
\newcommand{\tld}{\widetilde}
\newcommand{\floor}[1]{\lfloor #1 \rfloor}

\newcommand{\GL}{\mathrm{GL}}

\newcommand{\Fl}{\mathrm{Fl}} 

\newcommand{\tilW}{\til{W}}

\newcommand{\tilw}{\til{w}} 
\newcommand{\tilz}{\til{z}} 
 
\DeclareMathOperator{\Spec}{\mathrm{Spec}}

\DeclareMathOperator{\Spf}{\mathrm{Spf}}

\DeclareMathOperator{\Hom}{\mathrm{Hom}}

\DeclareMathOperator{\End}{\mathrm{End}}

\DeclareMathOperator{\Ind}{\mathrm{Ind}}

\DeclareMathOperator{\cind}{c\mathrm{\dash Ind}}
\DeclareMathOperator{\Lie}{\mathrm{Lie}}

\DeclareMathOperator{\Gal}{\mathrm{Gal}}

\DeclareMathOperator{\Ext}{\mathrm{Ext}}

\newcommand{\rhobar}{\overline{\rho}}
\newcommand{\rbar}{\overline{r}}
\newcommand{\mo}{{-1}}

\renewcommand{\ss}{{\mathrm{ss}}}

\newcommand{\gr}{\mathrm{gr}}
\newcommand{\Gr}{\mathrm{Gr}}

\newcommand{\et}{\normalfont{\text{{\'et}}}}

\newcommand{\JH}{\mathrm{JH}}

\newcommand{\cyc}{\mathrm{cyc}}

\newcommand{\ord}{\mathrm{ord}}

\newcommand{\red}{{\mathrm{red}}}
\newcommand{\reg}{{\mathrm{reg}}}

\newcommand{\nbl}{\nabla}
\newcommand{\bss}{{\backslash}}

\newcommand{\Adm}{{\mathrm{Adm}}}
\newcommand{\loccit}{\emph{loc.~cit.}}

\newcommand{\dash}{\text{-}}

\newcommand{\ix}[1]{^{(#1)}}

\newcommand{\pgma}{(\varphi,\Gamma)}

\newcommand{\pcp}{{\wedge_p}} 

\newcommand{\jj}{{j\in \cJ}}

\newcommand{\nblz}{{\nabla_0}}

\newcommand{\osig}{\overline{\sigma}}

\newcommand{\nbla}{\nabla_{\bfa}}
\newcommand{\nblat}{\nabla_{\bfa_\tau}}

\renewcommand{\ev}{\mathrm{ev}}
\renewcommand{\det}{\mathrm{det}}

\newcommand{\Ad}{\mathrm{Ad}}

\newcommand{\tils}{\Tilde{s}}

\newcommand{\dia}{\diamond}

\renewcommand{\SS}{\mathrm{SS}}
\newcommand{\KZ}{\mathrm{KZ}}

\newcommand{\mc}{\mathcal}
\newcommand{\mf}{\mathfrak}
\newcommand{\mbf}{\mathbf}

\newcommand{\Q}{\mathbf{Q}}

\newcommand{\Qp}{\mathbf{Q}_p}
\newcommand{\Qpbar}{\overline{\mathbf{Q}}_p}

\newcommand{\Z}{\mathbf{Z}}

\newcommand{\R}{\mathbf{R}}

\newcommand{\F}{\mathbf{F}}

\newcommand{\Fp}{\mathbf{F}_p}
\newcommand{\Fpbar}{\overline{\mathbf{F}}_p}

\newcommand{\G}{\mathbf{G}}

\newcommand{\T}{\mathbb{T}}
\newcommand{\A}{\mathbb{A}}
\renewcommand{\P}{\mathbb{P}}

\newcommand{\fm}{{\mf{m}}}

\newcommand{\fD}{{\mf{D}}}

\newcommand{\fM}{{\mf{M}}}

\newcommand{\cA}{{\mc{A}}}

\newcommand{\cC}{{\mc{C}}}

\newcommand{\cG}{{\mc{G}}}
\newcommand{\cH}{{\mc{H}}}
\newcommand{\cI}{{\mc{I}}}
\newcommand{\cJ}{{\mc{J}}}

\newcommand{\cO}{{\mc{O}}}
\newcommand{\cP}{{\mc{P}}}

\newcommand{\cX}{{\mc{X}}}

\newcommand{\bfa}{{\mbf{a}}}

\newcommand{\rmG}{{\mathrm{G}}}

\newcommand{\rmK}{{\mathrm{K}}}

\newcommand{\rmZ}{{\mathrm{Z}}}

\newcommand{\al}{\alpha}
\newcommand{\be}{\beta}

\newcommand{\Del}{\Delta}
\newcommand{\del}{\delta}

\newcommand{\veps}{\varepsilon}

\newcommand{\ka}{\kappa}
\newcommand{\Lam}{\Lambda}
\newcommand{\lam}{\lambda}
\newcommand{\sig}{\sigma}

\newcommand{\om}{\omega}
\newcommand{\Om}{\Omega}

\newcommand{\oom}{\overline{\omega}}

\title{Non-admissibility of some universal supersingular representations }

\author{Zachary Feng}
\address{University of Oxford, Mathematical Institute, Woodstock Road, Oxford OX2 6GG, UK}
\email{feng@maths.ox.ac.uk}

\author{Heejong Lee}
\address{Korea Institute for Advanced Study, 85 Hoegi-ro, Dongdaemun-gu, Seoul 02455, Republic of Korea}
\email{heejonglee@kias.re.kr}

\author{Ray Li}
\address{Department of Mathematics, University of Chicago, Chicago, Illinois 60637, USA}
\email{rayli@uchicago.edu}

\author{Vaughan McDonald}
\address{Stanford University, Building 380, Stanford, California 94305, USA}
\email{vkm@stanford.edu}

\author{Nischay Reddy}
\address{Department of Mathematics, University of Toronto, Toronto, Ontario M5S 2E4, Canada}
\email{nischay.reddy@mail.utoronto.ca}

\begin{document}

\maketitle

\begin{abstract}
    Let $K/\Qp$ be an unramified extension of degree $f$ with residue field $k$. Let $\sig$ be an irreducible representation of $\GL_n(k)$ over $\Fpbar$. For $n\ge 3$, we prove that the universal supersingular representation of weight $\sig$ is non-admissible and of infinite length when $\sig$ is sufficiently generic and satisfies certain technical conditions. This generalizes the previous results for $n=2$ and a non-trivial finite extension $K/\Qp$. 
    Our method employs a weight cycling argument together with recent progress on the Serre weight conjectures.
\end{abstract}

\section{Introduction}
\subsection{Background}

The mod $p$ local Langlands correspondence for $\GL_2(\Qp)$ gives a tight connection between 2-dimensional Galois representations of $\Gal(\ov{\Q}_p/\Qp)$ over $\Fpbar$ and admissible smooth representations of $\GL_2(\Qp)$ over $\Fpbar$ (see \cite{Col,Pas}). It is desirable to extend such a correspondence to $\GL_n(K)$ where $K/\Qp$ is a finite extension. When either $n\ge 3$ or $K\neq \Qp$, mod $p$ representations of $\GL_n(K)$ exhibit surprisingly complicated behaviors compared to the case of $\GL_2(\Qp)$. To explain such behaviors, we introduce the notion of \textit{supersingular} representations.

Let $\rmG:=\GL_n(K)$, $\rmK:=\GL_n(\cO_K)$, and $\rmZ$ be the center of $\rmG$. For an irreducible $\Fpbar$-representation $\sig$ of $\KZ$, we consider the \textit{universal supersingular representation of weight $\sig$}
\begin{align*}
    \SS(\sig):= \cind_{\KZ}^{\rmG}\sig/(T_1,\dots,T_{n-1})
\end{align*}
where $T_i$ is the Hecke operator of $\sig$ associated to the $i$th fundamental cocharacter. Then a \textit{supersingular representation} of $\rmG$ is equivalent to an irreducible admissible quotient of $\SS(\sig)$ for some $\sig$\footnote{Note that irreducible subquotients of $\SS(\sig)$ are not necessarily supersingular. In fact, the proof of our main result shows that there are infinitely many non-isomorphic non-supersingular irreducible subquotients of $\SS(\sig)$.}. \cite{AHHV} classified irreducible admissible mod $p$ representations of $p$-adic reductive groups in terms of supersingular representations, analogous to the case of complex representations. However, the classification of supersingular representations remains mysterious.

Let $n=2$. For $K=\Qp$, \cite{Breuil-SS} showed that $\SS(\sig)$ is already irreducible and admissible and thus classified supersingular representations of $\GL_2(\Qp)$. In stark contrast, when $K/\Qp$ is a non-trivial unramified extension, \cite{BP} showed that $\SS(\sig)$ admits infinitely many non-isomorphic irreducible quotients under a mild assumption on $\sig$. \cite{Schein, Hendel} showed that $\SS(\sig)$ is non-admissible if and only if $K\neq \Qp$.   Moreover, \cite{Schraen-SS, Wu-SS} showed that any irreducible quotient $\pi$ of $\SS(\sig)$ is \textit{not} of finite presentation when $n=2$ and $K\neq \Qp$, meaning that the kernel of $\SS(\sig)\epi \pi$ is not finitely generated as a $\Fpbar[\rmG]$-module. They also proved the non-admissibility of $\SS(\sig)$. In a different direction, \cite{Le-central-char} showed there exists an irreducible smooth mod $p$ representation of $\GL_n(K)$ that is non-admissible and does not have a central character for $n\ge 2$ and $K$ whose residue field is a non-trivial extension of $\Fp$ (see also, \cite{Le-non-adm,GS-non-adm,GLS-non-adm}).

\subsection{Result and method}
From now on, we assume that $K/\Qp$ is unramified of degree $f\ge 1$ with residue field $k$.

\begin{thm}[Theorem \ref{thm:main}]
    The universal supersingular representation $\SS(\sig)$ is non-admissible and of infinite length for $n\ge 3$ and sufficiently deep $\sig$ whose $p$-restricted highest weight is contained in a special $p$-restricted $p$-alcove.  
\end{thm}
This generalizes the previous non-admissibility results for $n=2$. The assumption on $K$ is required to apply the main result in \cite{LLMPQ-CL}. Our deepness assumption on $\sig$ requires $p$ to be sufficiently large. See \S\ref{subsec:sp-alc} for the definition of special $p$-restricted $p$-alcoves. The proportion of special ones among all $p$-restricted $p$-alcoves is $1-\left(\frac{n-2}{n-1}\right)^f$. See Remark \ref{rmk:non-special} for a discussion why the given problem is more difficult without this condition.

The methods in previous works are restricted to $\GL_2$. \cite{BP} extensively used the theory of diagrams (first introduced in \cite{Pas-diagram}) which is currently not available for groups other than $\GL_2$. \cite{Schein, Hendel} proved the non-admissibility of $\SS(\sig)$ by explicitly computing the pro-$p$ Iwahori invariant subspace of $\SS(\sig)$ for $\GL_2$ (for $K/\Qp$ unramified, this was first observed by Breuil; see \cite[Remarque 4.2.6]{Breuil-SS}). \cite{Schraen-SS, Wu-SS} used a similar argument and a result of Hu \cite{Hu-canonical-diag} relying on diagrams. \cite{GLS-non-adm, Le-central-char}, which studied different problems, also used the theory of diagrams to prove the result for $n=2$ and extended it to general $n$ using parabolic induction. This is why the residue field is assumed to be a non-trivial extension of $\Fp$ in \loccit

Our method is distinct from the previous ones in the sense that:
\begin{enumerate}
    \item we employ a weight cycling argument which applies to any split reductive group over any finite extension $K/\Qp$, and
    \item we make use of the moduli space of mod $p$ local Galois representations (the reduced part of the Emerton--Gee stack \cite{EGstack}) and the connection between mod $p$ local Galois representations and mod $p$ representations of $p$-adic groups given by globalization and the space of mod $p$ algebraic automorphic forms.
\end{enumerate}
Even though the proof of Theorem \ref{thm:main} works only when $n\ge 3$, a similar argument can prove the non-admissibility and the infinite length property of $\SS(\sig)$ for $n=2$ and $f\ge 2$. This is explained in Remark \ref{rmk:GL2}, which gives a concise illustration of our method. We view our method as a showcase for how Galois representations can guide the mod $p$ representation theory of $p$-adic groups. 

We explain our method in more detail. The weight cycling argument was first discovered by Buzzard for $\GL_{2}(\Qp)$. We use its most general form due to \cite{EGHweightcyc}. In fact, we reinterpret their result in the following form (Proposition \ref{prop:WC}): for $\sig$ an irreducible representation of $\KZ$ and $1 \le i \le n-1$, there exists a finite length smooth representation $C_i(\sig)$ of $\KZ$, which is defined as a parabolic induction of an irreducible $\Fpbar$-representation of a Levi subgroup in the finite group $\GL_n(k)$, with a surjective morphism
\begin{align*}
    \cyc_i: \cind_\KZ^{\rmG} C_i(\sig) \epi \cind_{\KZ}^\rmG \sig/(T_i).
\end{align*}
In particular, $\cind_\KZ^{\rmG} C_i(\sig)$ surjects onto $\SS(\sig)$. Let $\sig'$ be a Jordan--H\"older factor for $C_i(\sig)$. We want to show that the subquotient of $\SS(\sig)$ given by $\cind_{\KZ}^\rmG \sig'$ and $\cyc_i$ is non-admissible and of infinite length. Then this implies the same properties of $\SS(\sig)$. There are two difficulties. Firstly, we do not have a complete description of the set of Jordan--H\"older factors $\JH(C_i(\sig))$. Secondly, even if we know that some explicit $\sig'$ is in $\JH(C_i(\sig))$, understanding its image under $\cyc_i$ seems very hard. Due to these difficulties, applications of the weight cycling argument have been limited to $\GL_n$ for $n\le 3$.

This is where Galois representations come in. Let $G_K:=\Gal(\ov{K}/K)$ and $\rhobar: G_K \ra \GL_n(\Fpbar)$ be a continuous representation. Using a globalization of $\rhobar$, we can construct an admissible smooth $\Fpbar$-representation $\pi$ of $\rmG$ attached to $\rhobar$, and $\rhobar \mapsto \pi$ is a candidate for the mod $p$ local Langlands correspondence. By carefully choosing $\rhobar$, we can assume that $\sig$ is a modular weight of $\rhobar$. Furthermore, using the main result of \cite{lee-satake}, which reinterprets Hecke operators for $\sig$ in terms of Galois representations, we can take $\rhobar$ to be \textit{supersingular with respect to $\sig$} (Definition \ref{def:ss-rhobar}). By the mod $p$ local-global compatibility result in \cite{lee-satake}, we obtain a commutative diagram
\[
\begin{tikzcd}
    \cind_\KZ^{\rmG} C_i(\sig) \arrow[r, two heads] \arrow[rd, dashed] & \SS(\sig) \arrow[d, dashed] & \cind_{\KZ}^\rmG \sig \arrow[l, two heads] \arrow[ld] \\ 
    & \pi.
\end{tikzcd}
\]
Here, the right diagonal arrow exists because $\sig$ is a modular weight of $\rhobar$. It factors through $\SS(\sig)$ and induces the vertical dashed arrow by the supersingularity of $\rhobar$ with respect to $\sig$, which in turn induces the diagonal dashed arrow. In order to understand (part of) the left horizontal arrow, we study the left diagonal arrow. By Frobenius reciprocity, we get a non-zero morphism $C_i(\sig)\ra \pi|_{\KZ}$. In particular, there is at least one $\sig' \in \JH(C_i(\sig))$ that appears in the $\KZ$-socle of $\pi|_{\KZ}$. This is why this method is called weight cycling. 

In general, it is difficult to determine which $\sig'$ should appear in the $\KZ$-socle of $\pi$. We resolve this by choosing $\rhobar$ such that there is a unique $\sig' \in \JH(C_i(\sig))$ modular for $\rhobar$. Then this induces a non-zero morphism from $\cind_\KZ^\rmG \sig'$ to $\pi$.  
Let $\pi'$ be the subquotient of $\SS(\sig)$ given by $\cind_\KZ^\rmG \sig'$. In addition to above conditions on $\rhobar$, we require $\rhobar$ to be non-supersingular with respect to $\sig'$. In fact, we choose $\rhobar$ ordinary with respect to $\sig'$, i.e.~extension of characters whose restriction to the inertia subgroup is explicitly determined by $\sig'$. This implies that the image of $\cind_\KZ^\rmG \sig'$ in $\pi$, which factors through $\pi'$, is given by a principal series representation. By twisting the diagonal characters of $\rhobar$ by infinitely many unramified characters, we get infinitely many non-isomorphic principal series of fixed weight $\sig'$ as quotients of $\pi'$.  
This implies the desired non-admissibility and infinite length of (a subquotient of) $\SS(\sig)$ (see Lemma \ref{lem:non-adm}).

What remains is to find (a family of) $\rhobar$ satisfying the above conditions. Let $\cX_{n,\red}$ be the reduced Emerton--Gee stack constructed in \cite{EGstack}. It can be viewed as the moduli space of $n$-dimensional mod $p$ representations of $G_K$. It is known that there is a bijection between irreducible representations of $\GL_n(k)$ and irreducible components in $\cX_{n,\red}$. \cite{LLMPQ-CL} studied explicit local charts of certain subspaces in $\cX_{n,\red}$. The regular colength one local charts in \loccit~have two irreducible components corresponding to $\sig$ and $\sig'$. Then $\rhobar$ in the intersection of the two components has $\sig$ and $\sig'$ as its modular weights (Proposition \ref{prop:SWC}). For the ordinarity of $\rhobar$ with respect to $\sig'$, we further intersect with the ordinary locus in the component corresponding to $\sig'$. Showing the non-emptiness of this triple intersection is the most technical part of our work (Proposition \ref{prop:int}). Using local model theory developed in \cite{LLLMlocalmodel}, we reduce this to computing intersections of certain translated affine Schubert cells. Then the specialness of the $p$-alcove ensures that our $\rhobar$ is supersingular with respect to $\sig$. Our result also shows that the supersingular locus in the irreducible component for $\sig$ is of codimension one when the ($p$-restricted) highest weight of $\sig$ belongs to a special $p$-alcove (Corollary \ref{cor:ss}). This exhibits a striking contrast compared to the lowest alcove case.

\subsection{Acknowledgments} This project was initiated at the 2025 Arizona Winter School: Representation theory of $p$-adic groups, as part of Florian Herzig's project groups. We thank the organizers for coordinating the event and for providing a great working environment. We are especially grateful to Florian Herzig for suggesting this problem and for his guidance. We also thank Madison Delmoe for her participation in the project during the winter school. We benefited from many helpful discussions with Daniel Le and are pleased to acknowledge his contribution. In particular, the reinterpretation of weight cycling and Lemma \ref{lem:non-adm}, which is crucial for the non-admissibility result, grew out of the discussion between H.L.~and Daniel Le.

Z.F. acknowledges the support of the Natural Sciences and Engineering Research Council of Canada (NSERC), [ref. no. 577979-2023]. H.L.~was supported by the AMS--Simons Travel Grant and by the KIAS Individual Grant (HP103001) at the Korea Institute for Advanced Study. He also thanks Purdue University for its hospitality during his visit in Fall 2025. During part of this work, V.M. was supported by NSF Graduate Research Fellowship grant DGE-1656518.

\subsection{Notation}
For a field $F$, we fix its separable closure $\ov{F}$ and define $G_F := \Gal(\ov{F}/F)$. Let $K/\Qp$ be the unramified extension of degree $f$ with ring of integers $\cO_K$ and residue field $k$. Let $E\subset \Qpbar$ be a sufficiently large finite extension with ring of integers $\cO$, uniformizer $\varpi\in \cO$, and residue field $\F$. Let $\varphi$ be the arithmetic absolute Frobenius automorphism on $K$. We fix an embedding $\sig_0 : K \mono E$ and write $\sig_j:= \sig_0\circ \varphi^{-j}$. Let $\cJ$ be the set of embeddings of $K$ into $E$. We identify $\Z/f\Z \simeq \cJ$ by $j\mapsto \sig_j$. We choose $\pi\in \ov{K}$ such that $\pi^{p^f-1}=-p$. Let $\om_K:G_K \ra \cO_K^\times$ be Serre's fundamental character given by $g(\pi)=\om_K(g)\pi$ for $g\in G_K$. We write $\om_{K,\sig_j}:= \sig_j\circ \om_K$ and $\oom_{K,\sig_j} = \om_{K,\sig_j} \mod \varpi$. We normalize Hodge--Tate weights so that the $p$-adic cyclotomic character $\veps$ has Hodge--Tate weights 1 at every embedding.

Let $G$ be a split connected reductive group over $\cO_K$. Except in \S\ref{sec:rep}, we will take $G=\GL_n$. We assume throughout that $G$ has simply connected
derived subgroup. Let $T\subset B \subset G$ denote a maximal torus and a Borel subgroup and $U$ be the unipotent radical of $B$. We write $\ov{B}=T\ov{U}$ for the opposite Borel subgroup. Following the standard notations, we denote by $X^*(T)$, $X_*(T)$, $\Del$, $\Del^\vee$, $\Phi$, $\Phi^+$, $\Phi^-$, $\Lam\subset X^*(T)$ for the groups of characters, cocharacters, the sets of fundamental roots, fundamental coroots, roots, positive roots, negative roots, and the root lattice generated by $\Delta$, respectively. We fix $\eta\in X^*(T)$ satisfying $\RG{\eta,\al^\vee}=1$ for all $\al\in \Delta$. When $G=\GL_n$, we take $\eta=(n-1,\dots,1,0)$. We write $\rmG:=G(K)$, $\rmK :=G(\cO_K)$, and $\rmZ$ for the center of $\rmG$. For $\al\in \Phi$, we write $u_\al: \G_a \ra G$ for the morphism associated to the root subgroup of $\al$.

We write $\Ind$ and $\cind$ for the usual induction and compact induction, respectively. Let $H$ be a topological group with a compact open subgroup $H'$. For $h\in H$, $H'$-representation $V$, and $v\in V$, we write $[h,v]\in \cind_{H'}^H V$ for the function supported on $H'h^\mo$ and maps $h^\mo$ to $v$. Note that $h'[h,v]=[h'h,v]$ for $h'\in H$.

For $A\in \GL_n(R)$ and $B\in M_n(R)$, we write $\Ad_A(B) := A B A^\mo$.

\section{Representation theory}\label{sec:rep}

\subsection{Preliminaries} We review some background material. We closely follow \cite[\S2]{LLLMlocalmodel}.

\subsubsection*{Affine Weyl groups} Let $W$ be the Weyl group associated to $(G,T)$. We define $W_a:=\Lam \rtimes W$ and $\tilW:= X^*(T) \rtimes W$. Let $\cA$ be the set of alcoves in $X^*(T)\otimes \R$. Let $A_0\in \cA$ be the dominant base alcove. Then $\tilW$ acts transitively on $\cA$. We define $\Om:= \{ \tilw\in \tilW \mid \tilw A_0 = A_0\}.$ Then $\tilW = \Om \rtimes W_a$. We denote by $\ell$ the Coxeter length function on $W_a$. The Bruhat order on $\cA$ given by the choice of $A_0$ induces an ordering on $W_a$. We also have $\uparrow$-order on $\cA$ which induces the $\uparrow$-order on $W_a$. For $\tilw,\tilw' \in W_a$ and $\del \in \Om$, we write $\tilw\del \le \tilw'\del$ if $\tilw \le \tilw'$ and similarly for the $\uparrow$-order. We extend $\ell$ to $\tilW$ by $\ell(\tilw \del) := \ell(\tilw)$. 
 
For $\al\in \Phi^+$ and $m\in \Z$, we define
\begin{align*}
    H_{\al}^{(m,m+1)} &:= \{ \lam\in X^*(T)\otimes \R \mid m<\RG{\lam,\al^\vee}<m+1 \} \\
H_{\al,m} &:= \{ \lam\in X^*(T) \mid \RG{\lam,\al^\vee} = m \}.
\end{align*}
For $\tilw\in \tilW$ and $\al\in \Phi^+$, we define $m_{\tilw,\al}\in \Z$ by the condition $\tilw A_0\subset H_{\al}^{(m_{\tilw,\al},m_{\tilw,\al+1})}$. Since $\ell(\tilw)$ is equal to the number of hyperplanes $H_{\al,m}$ separating $A_0$ and $\tilw A_0$, we have $\ell(\tilw) = \sum_{\al\in \Phi^+} \abs{m_{\tilw,\al}}$.

We say that an alcove $A\in \cA$ is \textit{restricted} (resp.~\textit{regular}) if it is contained in (resp.~not contained in) $H_{\al}^{(0,1)}$ for all $\al\in \Del$. We say that $\tilw \in \tilW$ is restricted (resp.~regular) if $\tilw A_0$ is. Let $\cA_1$ be the set of restricted alcoves. We define $\tilW_1\subset \tilW$ to be the subset of restricted $\tilw$. We also define $X^0(T)\subset X^*(T)$ to be the subset of $\lam$ such that $\RG{\lam,\al^\vee}=0$ for all $\al\in \Del$. Then there is a bijection between
\begin{align*}
    W &\risom \tilW_1/X^0(T) \\
    w &\mapsto w^\dia := t_{\nu_w} w
\end{align*}
where $\nu_w\in X^*(T)$ is a character unique up to $X^0(T)$ such that $t_{\nu_w}w \in \tilW_1$. We often write $w^\dia$ for its lift in $\tilW_1$.

\begin{rmk}\label{rmk:superadd}
    For $\tilw=t_\nu w\in \tilW$, $\al\in \Phi^+$, and $x\in A_0$, we have
    \begin{align*}
         m_{\tilw,\al} = \floor{\RG{\tilw(x),\al^\vee}} = \RG{\nu,\al^\vee} - \del_{w^\mo(\al)^\vee<0}. 
    \end{align*}

    We can view $\al\mapsto m_{\tilw,\al}$ as a function $\Phi \ra \Z$. Then it satisfies the \textit{optimal superadditivity}: for $\al_1,\al_2\in \Phi$ such that $\al_1+\al_2\in \Phi$, 
    \begin{align*}
        m_{\tilw,\al_1} + m_{\tilw,\al_2} \le m_{\tilw,\al_1+\al_2} \le m_{\tilw,\al_1} + m_{\tilw,\al_2} +1.
    \end{align*}
    Note that $\tilw$ is restricted if and only if $m_{\tilw,\al}=0$ for all $\al\in \Del$.
\end{rmk}

\begin{rmk}\label{rmk:def-S}
    Let $G=\GL_n$ and $S\subset W$ be the subgroup generated by the $n$-cycle $(12\dots n)$. Then $S$ is the preimage of $\Omega$ under $w\mapsto w^\dia$. Note that $w_1^\dia w_2^\dia \neq (w_1w_2)^\dia$ for $w_1,w_2\in W$ in general, but the equality holds if $w_2\in S$. Then $w\mapsto w^\dia$ defines a natural bijection $W/S \simeq \cA_1$. In particular, $\#\cA_1 = (n-1)!$.
\end{rmk}

For $\lam\in X^*(T)$, we define the admissible set introduced by Kottwitz and Rapoport
\begin{align*}
    \Adm(\lam) = \cup_{w\in W} \{ \tilw \in \tilW \mid \tilw \le t_{w(\lam)}\}.
\end{align*}
We denote by $\Adm^{\reg}(\lam) \subset \Adm(\lam)$ the subset of $\tilw \in \Adm(\lam)$ that is regular.

We also have the dual extended affine Weyl group $\tilW^\vee := X_*(T^\vee)\rtimes W$. We have a bijection  $\tilW \simeq \tilW^\vee$ given by $\tilw = t_\nu w \mapsto \tilw^*:= w^\mo t_{\nu}$. We denote by $\Adm^\vee(\lam)$ the image of  $\Adm(\lam)$ under this bijection. We call an element $\tilw \in \Adm(\lam)$ (and $\tilw^*$) \textit{extremal} (resp.~\textit{of colength one}) if it attains the maximal length,~i.e.~$\tilw=t_{w(\lam)}$ for some $w\in W$ (resp.~if $\ell(\tilw)=\ell(t_{\eta})-1$). We say that $\tilw$ is \textit{of colength at most one} if $\ell(\tilw)\ge \ell(t_{\eta})-1$.

To ease the notation, we write $w^{-\dia} := (w^\dia)^\mo$, $\tilw^{-*} := (\tilw^\mo)^*$, $w^{\dia*} := (w^\dia)^*$, and $w^{-\dia*}:= (w^{-\dia})^*$ for $w\in W$ and $\tilw\in \tilW$. Caution: $(w^\mo)^\dia \neq w^{-\dia}$ in general.

We also define a ($\eta$-shifted) \textit{$p$-alcove} $C$ by 
\begin{align*}
    C:=\{ \lam \in X^*(T)\otimes \R \mid pm_\al < \RG{\lam+\eta, \al^\vee} < p(m_\al+1) \forall \al\in \Phi^+\}
\end{align*}
for some $m_{\al}\in \Z$. We write $C_0$ for the dominant base $p$-alcove and $\cC$ for the set of $p$-alcoves. For $\tilw=t_\nu w\in \tilW$ and $\lam\in X^*(T)\otimes \R$, we define the \textit{dot action} $\tilw \cdot \lam := p\nu + w(\lam+\eta)-\eta$. Then $\tilW$ acts on $\cC$ via the dot action. 

We often consider a $f$-fold product of $\tilW$, denoted by $\tilW^\cJ$. Let $\pi: \tilW^\cJ \ra \tilW^\cJ$ be an automorphism given by $\pi(\tilw)_j = \tilw_{j+1}$ for $\tilw=(\tilw_j)_{\jj}$. For $\jj$, we write $\sig_j$ for the embedding $\tilW\mono \tilW^\cJ$ into $j$th coordinate. We also write $\pi$ (resp.~$\sig_j$) for its restrictions to $X^*(T)^\cJ$ and $W^\cJ$ (resp.~$X^*(T)$ and $W$).

\subsubsection*{Serre weights} A \textit{Serre weight} of $G(k)$ is an irreducible $\F$-representation of $G(k)$. We can classify Serre weights of $G(k)$ using the restricted group $\Res_{k/\Fp}G_{/k}$. Its character group is given by $f$-fold product of $X^*(T)$. Define the set of $p$-restricted dominant weights
\begin{align*}
    X_1^*(T):= \CB{{\lam} \in X^*(T) \mid 0 \le \RG{{\lam},\al^\vee} \le p-1, \  \forall \al^\vee \in \Del^{\vee}}.
\end{align*}
Then we have a bijection (see \cite[Lem.~9.2.4]{GHS-JEMS-2018MR3871496})
\begin{align*}
    \frac{X_1^*(T)^\cJ}{(p-\pi)X^0(T)^\cJ} &\simeq \CB{\text{Serre weights of $\GL_n(k)$}}/\simeq \\
   {\lam}=(\lam_j)_{\jj}  &\mapsto F({\lam}).
\end{align*}
A \textit{lowest alcove presentation} of $\sig$ is a pair $(\tilw, \om)$ of $\tilw\in \tilW_1^\cJ$ and $\om-\eta \in C_0^\cJ \cap X^*(T)^\cJ$ such that $\sig = F(\pi^\mo(\tilw)\cdot (\om-\eta))$. For $m\in \Z$, we say that $\sig$ is \textit{$m$-deep} if $m<\RG{\om,\al^\vee}<p-m$ for all $\al\in \cup_{\jj}\sig_j(\Phi^+)$.

For $\lambda\in X_1^*(T)$ and a standard Levi subgroup $M\subset G$, we write $F_M(\lam)$ for the Serre weight of $M(k)$ by applying the above formalism for $M$. When $\sig=F(\lam)$, we call $\lam$ the \textit{$p$-restricted highest weight of} $\sig$. By \cite[Lemma 2.3]{HerzigDuke}, $\sig^{U(k)}$ is a character of $T(k)$ given by $F_T(\lam)$, and we call it the  \textit{highest weight of $\sig$}.

\subsubsection*{Tame inertial types and the inertial local Langlands} A \textit{tame inertial type} over $R$ (for $R=E$ or $\F$) is a $G^\vee(R)$-conjugacy class of homomorphisms $I_K \ra G^\vee(R)$ that are tamely ramified. Since it is valued in $T^\vee(R)$ up to conjugation, tame inertial types over $E$ and over $\F$ are in bijection under mod $p$ reduction and Teichm\"uller lifting. We refer to \cite[\S2.4]{LLLMlocalmodel} for a detailed discussion on tame inertial types. We simply mention that a tame inertial type $\tau$ (either over $E$ or over $\F$) can be given by a pair $(s,\mu)\in W^\cJ \times X^*(T)^\cJ$, in which case we write $\tau = \tau(s,\mu+\eta)$. For $\tilw=t_\nu w \in \tilW$, we have $\tau(s,\mu+\eta)\simeq \tau(s',\mu'+\eta)$ where
\begin{align}\label{change-of-LAP}
    (s',\mu')= \tilw\cdot (s,\mu) := (ws\pi(w)^\mo, \tilw\cdot \mu - ws\pi(w)^\mo \pi(\tilw)(0)).
\end{align}
When $\mu\in C_0^\cJ$, we say that $(s,\mu)$ is a \textit{lowest alcove presentation} of $\tau$, and $\tau$ is \textit{$m$-generic} if $m<\RG{\mu+\eta,\al^\vee}<p-m$ for all $\al\in \cup_{\jj}\sig_j(\Phi^+)$.  We write $\tilw(\tau):=t_{\mu+\eta}s$ and $\tilw^*(\tau) := \tilw(\tau)^*$. For another tame inertial type $\rhobar$ with a given lowest alcove presentation, we write $\tilw(\rhobar,\tau) := \tilw(\tau)^\mo \tilw(\rhobar)$ and $\tilw^*(\rhobar,\tau) := \tilw(\rhobar,\tau)^*$.  

Let $G=\GL_n$. For a tame inertial type $\tau$ over $E$, we denote by $\tau \mapsto \sig(\tau)$ the inertial local Langlands correspondence \cite[Theorem 2.5.4]{LLLMlocalmodel}. Here, $\sig(\tau)$ is a Deligne--Lusztig representation, and it is a genuine irreducible representation of $G(k)$ over $E$ when $\tau$ admits a 1-generic lowest alcove presentation (see \S2.3 in \loccit).

\begin{example}\label{ex:PS}
    Let $\sig:=F(\mu)$ with $\mu\in X^*_1(T)^\cJ$. For $\tau=\tau(e,\mu)$, $\sig(\tau)$ is the induced representation $\Ind_{B(k)}^{G(k)} [\tld{\mu}]$ where $\tld{\mu} := \sig^{U(k)}$ is the highest weight of $\sig$.
\end{example}

Let $\osig(\tau)$ be the mod $p$ reduction of a $G(k)$-stable $\cO$-lattice in $\sig(\tau)$ and $\JH(\osig(\tau))$ be its set of irreducible constituents.  When $\tau$ is $(n-1)$-generic, we have an explicit description of $\JH(\osig(\tau))$ \cite[Theorem 1.1]{LLLM-DL} (also, see \cite[Proposition 2.3.7]{LLLMlocalmodel}).

\subsection{Weight cycling argument}
In this subsection, we allow $K/\Qp$ to be ramified and denote its uniformizer by $\pi_K$. We introduce the weight cycling argument. First, we recollect a few results on the Hecke algebra of a Serre weight.

For a dominant or antidominant cocharacter $\mu\in X_*(T)$, we define $P_\mu = M_\mu N_\mu$ to be the parabolic subgroup with Levi factor $M_\mu$ and unipotent radical $N_\mu$ given by the condition that $\al$-entry vanishes for all roots $\al$ such that $\RG{\al,\mu}>0$. Note that $P_{-\mu}$ is the opposite parabolic to $P_\mu$. For a Serre weight $\sig$, the natural map $\sig^{N_{-\mu}(k)} \mono \sig \epi \sig_{N_\mu(k)}$ is an isomorphism of $M_\mu(k)$-representations, and $\sig^{N_{-\mu}(k)}$ is a Serre weight of $M_\mu(k)$ \cite[Lemma 2.3]{HerzigDuke}.

Let $\cH(\sig):= \End_{\rmG}(\cind_\rmK^\rmG \sig)$ be the Hecke algebra of $\sig$. It can be identified with the algebra of $\rmK$-bi-equivariant functions $\phi: \rmG \ra \End_{\F}(\sig)$ with compact support. For $\mu\in X_*(T)^-$, let $T_{\mu}$ be the Hecke operator supported on $\rmK t \rmK$ where $t=\mu(\pi_K)$ and $T_\mu(t)$ is given by
\begin{align*}
    \sig \epi \sig_{N_\mu(k)} \simeq \sig^{N_{-\mu}(k)} \mono \sig.
\end{align*}
By \cite[Theorem 1.2]{HerzigSatake}, $\cH(\sig)\simeq \F[x_1,\dots, x_{n-1},x_n^\pm]$ under which $x_i$ is identified with $T_{-\om_i}$ for $1\le i \le n-1$ and $x_n$ is identified with $T_{(-1,-1,\dots,-1)}$. 

Recall that for $v\in \sig$, $[1,v] \in \cind_\rmK^\rmG \sig$ is the function supported on $\rmK$ and maps $1\in \rmK$ to $v$. For $\varphi\in \cH(\sig)$, we have (see the proof of Proposition 2.9 in \cite{HerzigAWS}) 
\begin{align}\label{eqn:convolution}
    \varphi( [1,v]) = \sum_{\gamma\in \rmK\backslash\rmG} \gamma^\mo[1,\varphi(\gamma)(v)].
\end{align}

\begin{defn}
    For a Serre weight $\sig$ and an antidominant fundamental cocharacter $\mu$, we define
    \begin{align*}
        C_\mu(\sig) := (\Ind_{P_{\mu}(k)}^{G(k)} \sig^{N_{-\mu}(k)})/\sig.
    \end{align*}
    Here, the embedding of $\sig$ into $\Ind_{P_\mu(k)}^{G(k)} \sig^{N_{-\mu}(k)}$ is induced by $\sig_{N_{\mu}(k)}\simeq \sig^{N_{-\mu}(k)}$ and the Frobenius reciprocity.
\end{defn}

\begin{prop}[Weight cycling]\label{prop:WC}
    Let $\sig$ be a Serre weight and $\mu$ be an antidominant fundamental cocharacter. Then there is a commutative diagram
    \[
    \begin{tikzcd}
        & \cind_\rmK^\rmG \Ind_{P_\mu(k)}^{G(k)} \sig^{N_{-\mu}(k)} \arrow[rd, two heads, dashed] & \\
        \cind_\rmK^\rmG \sig \arrow[ru, hook] \arrow[rr, "T_\mu"] & & \cind_\rmK^\rmG \sig
    \end{tikzcd}
    \]
    where all morphisms are $\rmG$-equivariant and the left diagonal arrow is induced by $\sig \mono \Ind_{P_\mu(k)}^{G(k)} \sig^{N_{-\mu}(k)}$. In particular, there is a $\rmG$-equivariant surjection
    \begin{align*}
        \cyc_\mu : \cind_\rmK^\rmG C_\mu(\sig) \epi \cind_\rmK^\rmG  \sig/T_\mu.
    \end{align*}
\end{prop}
When $G=\GL_n$ and $\mu=-\om_i$, we write $C_i(\sig)=C_\mu(\sig)$ and $\cyc_i = \cyc_\mu$.

\begin{proof}
    The following argument is essentially due to Emerton--Gee--Herzig. Let $\pi$ be a smooth $\rmG$-representation over $\F$ with a non-zero morphism $f: \sig \ra \pi|_\rmK$. There is a commutative diagram of $\rmK$-representations with a non-zero map $\tilde{\theta}$
    \[
    \begin{tikzcd}
        & \Ind_{P_\mu(k)}^{G(k)} \sig^{N_{-\mu}(k)} \arrow[rd, dashed, "\tilde{\theta}"] & \\
        \sig \arrow[ru, hook, "\iota"] \arrow[rr, "T_\mu(f)"] & & \pi|_\rmK.
    \end{tikzcd}
    \]
    For $\pi=\cind_\rmK^\rmG \sig$ and $f$ the natural embedding of $\sig$ into $\cind_\rmK^\rmG \sig$, we get the diagram in the Proposition using Frobenius reciprocity.

    Let $t= \mu(\pi_K)$. The right diagonal arrow $\tilde{\theta}$ is induced by $\theta$ in the following diagram
    \[
    \begin{tikzcd}
        \sig \arrow[rrr, bend right=16, "T_{\mu}(t)"']  \arrow[r, two heads] & \sig_{N_{\mu}(k)}  & \sig^{N_{-\mu}(k)} \arrow[rrr, bend right=20, "\theta"'] \arrow[l, "\simeq"'] \arrow[r, hook] & \sig \arrow[r, "f"] & \pi \arrow[r, "t^\mo"] & \pi.
    \end{tikzcd}
    \]
    Since $T_\mu(t)$ is $\rmK$-bi-equivariant and $f$ is $\rmK$-linear, $\theta$ is $\cP_\mu:=\rmK\cap (t^\mo \rmK t)$-linear. Note that $\cP_\mu$ is the parahoric subgroup given by the inverse image of $P_\mu(k)$ under $\rmK \epi G(k)$. Since $t^\mo f$ is injective, $\theta$ is non-zero, and thus $\tilde{\theta}$ is non-zero.

    We denote the projection $\sig \epi \sig_{N_{\mu}(k)}  \simeq \sig^{N_{-\mu}(k)}$ by $v\mapsto \ov{v}$. We have
    \begin{align*}
        \iota(v) = \sum_{k\in P_\mu(k)\backslash G(k)}k^\mo[1,\ov{kv}].
    \end{align*}
    Then $\tilde{\theta}\circ \iota = T_\mu(f)$ by a direct computation using \eqref{eqn:convolution}:
    \begin{align*}
        \tilde{\theta}\circ\iota(v) = \sum_{k\in P_\mu(k)\backslash G(k)}k^\mo \theta(\ov{kv}) = \sum_{k\in P_\mu(k)\backslash G(k)}k^\mo t^\mo f(T_\mu(t)(kv)) = \sum_{g\in \rmK\backslash \rmK t \rmK}g^\mo  f(T_\mu(g)(v)) = T_\mu(f)(v).
    \end{align*}

    Finally, we show that the map
    \begin{align*}
        \cind_\rmK^\rmG \Ind_{P_\mu(k)}^{G(k)} \sig^{N_{-\mu}(k)} \ra \cind_\rmK^\rmG \sig
    \end{align*}
    is surjective. Let $\pi$ be its cokernel. If $\pi$ is non-zero, we have a non-zero map $\sig \ra \pi|_\rmK$. Thus, we get a non-zero map from $\cind_\rmK^\rmG \Ind_{P_\mu(k)}^{G(k)} \sig^{N_{-\mu}(k)}$ to $\pi$. However, by its construction, it factors through the composition
    \begin{align*}
        \cind_\rmK^\rmG \Ind_{P_\mu(k)}^{G(k)} \sig^{N_{-\mu}(k)} \ra \cind_\rmK^\rmG \sig \epi \pi
    \end{align*}
    which is zero. This is a contradiction.
\end{proof}

\begin{rmk}
    Proposition \ref{prop:WC} is essentially a dual statement of \cite[Proposition 2.3.1]{EGHweightcyc}. Since $\pi$ is an arbitrary smooth representation of $\rmG$ over $\F$ in \loccit, their statement applied to $V=\sig^\vee$ can be viewed as a commutative diagram of functors
    \[
    \begin{tikzcd}
        & (\Ind_{P_{w_0(\mu)}(k)}^{G(k)} (\sig^\vee)^{N_{w_0(\mu)}(k)} \otimes_\F -)^{\rmK} \arrow[rd] & \\
        (\sig^\vee \otimes_\F -)^\rmK \arrow[ru, dashed, hook] \arrow[rr, "T_\mu"] & & (\sig^\vee \otimes_\F -)^\rmK.
    \end{tikzcd}
    \]
    Note that \loccit~implicitly identifies $\cH(\sig)$ and $\cH(\sig^\vee)^{\mathrm{op}}$ (the opposite algebra) using the anti-isomorphism
    \begin{align*}
        \iota: \cH(\sig) &\simeq \cH(\sig^\vee)^{\mathrm{op}} \\
        \varphi &\mapsto \iota(\varphi)(g)=\varphi(g^\mo)^\vee
    \end{align*}
    where $\varphi(g^\mo)^\vee$ acts on $\sig^\vee$ via precomposition-by-$\varphi(g^\mo)$. Then $\iota$ maps $T_\mu\in \cH(\sig)$ to $T_{-w_0(\mu)}\in \cH(\sig^\vee)^{\mathrm{op}}$. This explains the presence of $w_0(\mu)$ in the top induced representation. Write $M=M_{\mu}$ and $M'=M_{-w_0(\mu)}$. Then we have 
    \begin{align*}
        \Ind_{P_{w_0(\mu)}(k)}^{G(k)} (\sig^\vee)^{N_{w_0(\mu)}(k)} \simeq \Ind_{P_{w_0(\mu)}(k)}^{G(k)} F_{M'}(-w_0(\lambda)) \simeq (\Ind_{P_{w_0(\mu)}(k)}^{G(k)} F_{M'}(w_{0,M'}w_0(\lambda)))^\vee \simeq (\Ind_{P_{\mu}(k)}^{G(k)} F_{M}(\lambda))^\vee.
    \end{align*}
    The second isomorphism follows from the commutativity between parabolic induction and duality for finite groups. The third isomorphism follows from $F_{M'}(w_{0,M'}w_0(\lambda))\ix{w_0} \simeq F_{M}(\lambda)$ where we use $w_{0,M}=w_0w_{0,M'}w_0$. Note that $F_M(\lambda)=\sig^{N_{-\mu}(k)}$. Since $(\sig^\vee \otimes_\F -)^\rmK \simeq \Hom_{\rmG}(\cind_\rmK^\rmG \sig, -)$ and similarly for $\sig^\vee$ replaced by $(\Ind_{P_\mu(k)}^{G(k)} \sig^{N_{-\mu}(k)})^\vee$, Proposition \ref{prop:WC} follows from Yoneda's lemma.
\end{rmk}

Note that $C_\mu(\sig)$ is usually non-semisimple. We are not aware of a general result describing the socle filtration of $C_\mu(\sig)$, or even the irreducible constituents of it. Still, we have the following partial result.

\begin{lem}\label{lem:upperbound}
    Let $\sig$ and $\mu$ be as in Proposition \ref{prop:WC}. Let $\tld{\lambda}:=\sig^{U(k)}$ be the highest weight. Then there is a $G(k)$-equivariant injection
    \begin{align*}
       C_\mu(\sig) \mono (\Ind_{\ov{B}(k)}^{G(k)}\tld{\lam})/\sig. 
    \end{align*}
    In particular, $\sig \notin \JH(C_\mu(\sig)) \subset \JH(\Ind_{\ov{B}(k)}^{G(k)}\tld{\lam})$.
\end{lem}
\begin{proof}
    Note that the socle of $\Ind^{M_{\mu}(k)}_{\ov{B}(k)\cap M_\mu(k)} \tld{\lam}$ contains $F_{M_{\mu}}(\lambda) = \sig^{N_{-\mu}(k)}$. By the transitivity of parabolic induction, we get $\Ind_{P_\mu(k)}^{G(k)} \sig^{N_{-\mu}(k)} \mono \Ind_{\ov{B}(k)}^{G(k)}\tld{\lam}$. This induces the claimed injection. The claim on Jordan--H\"older factors holds because $\sig$ appears in $\JH(\Ind_{\ov{B}(k)}^{G(k)}\tld{\lam})$ with multiplicity one \cite[Theorem 1.3]{LLLM-DL}.
\end{proof}

For $G=\GL_n$, we will show that certain $\sig'$ is a Jordan--H\"older factor of $C_i(\sig)$ with multiplicity one for any $1 \le i\le n-1$ using a local-global compatibility type argument; see the proof of Theorem \ref{thm:main}.

We finish this subsection with lemmas on non-admissibility.

\begin{lem}\label{lem:subquot}
    Let $\pi$ be a smooth $\F[\rmG]$-module and $\pi'$ be a subquotient of $\pi$. If $\pi'$ is non-admissible, then so is $\pi$. 
    
\end{lem}
\begin{proof}
    See {\cite[Proposition 2.2.13]{Emerton-ordinaryI}} or \cite[Th\'eor\`eme 2]{Vigneras-torsion}.
\end{proof}

\begin{lem}\label{lem:non-adm}
    Let $\pi$ be a smooth $\rmG$-representation over $\F$. If $\pi$ admits infinitely many non-isomorphic irreducible quotients, then $\pi$ is non-admissible. 
\end{lem}
\begin{proof}
    By \cite[Theorem 2.1.2]{Emerton-ordinaryI}, $\F\DB{G(\cO_K)}$ is Noetherian. By Lemma 2.2.11 in \loccit, $\pi$ is admissible if and only if its $\F$-linear dual $\pi^\vee$ is finitely generated over $\F\DB{G(\cO_K)}$. We show that $\pi^\vee$ does not satisfy the ascending chain condition.

    Let $\{\pi_i\}_{i\in I}$ be the set of non-isomorphic irreducible quotients of $\pi$ indexed by $i\in I$. Then $\pi_i^\vee \subset \pi^\vee$. By \cite[Lemma 2.2.7]{Emerton-ordinaryI}, $\pi_i^\vee$'s (equipped with $\rmG$-action) for $i\in I$ are pairwise non-isomorphic. Thus, $\pi^\vee$ contains $\oplus_{i\in S}\pi_i^\vee$ for any finite subset $S\subset I$. Thus, $\pi^\vee$ does not satisfy the ascending chain condition.
    \end{proof}

\subsection{Special alcoves}\label{subsec:sp-alc}
We introduce a condition on alcoves required for our main result. We fix $\sig=F_{(w^\dia,\om)}$ a Serre weight. We need an auxiliary Serre weight $\sig'= F_{(u^\dia,\om)}$ such that
\begin{itemize}
    \item $w=s_\al u$ for some $\al=\sig_{j_0}(\al_{i_0k_0})\in \cup_{\jj}\sig_j(\Phi^+)$;
    \item $u^\dia \uparrow w^\dia$ (in particular, $\sig' \uparrow \sig$); and
    \item $\ell(w^\dia) = \ell(u^\dia)+1$.
\end{itemize}

The pair $(\sig,\sig')$ will be used to define certain explicit local charts of potentially crystalline Emerton--Gee stacks in \S\ref{subsec:chart}. For the application to our main result, we need an extra condition.
\begin{defn}
    We say that $w^\dia$ is \textit{special} if there exists $u^\dia$ and $s_{\al}$ as above such that $\al_{i_0k_0}\notin \Phi^+_M$ for any proper standard Levi subgroup $M\subset G$. We say that a $p$-restricted $p$-alcove $C = (C_j)_{\jj}$ is \textit{special} if $C=w^\dia \cdot C_0^\cJ$ for a special $w^\dia$.
\end{defn}

\begin{rmk}\label{rmk:special}
Let $G=\GL_n$. 
    \begin{enumerate}
        \item For $u$ and $w$ as above, we define
    \begin{align*}
        \tilw := w^{-\dia} \tilw_h^\mo w_0 u^\dia = w^\mo t_{\eta + \nu_{u}- \nu_w} u.
    \end{align*}
    By the conditions on $w^\dia$ and $u^\dia$,  $w^\dia = s_{\al,m} u^\dia$ for some $m\in \Z_{>0}$ and $m_{w^{\dia},\al} = m_{u^\dia,\al}+1$.  
This implies that either
    \begin{itemize}
        \item [(a)] $\nu_w=\nu_u$ and $ \RG{\nu_u,\al^\vee} = \RG{\nu_w, \al^\vee} = m$; or
        \item [(b)] $\nu_w = \nu_u + \al$ and $\RG{\nu_u,\al^\vee} = m-1,  \RG{\nu_w, \al^\vee} = m+1$.
    \end{itemize}
    Therefore, $\tilw = w^\mo t_{\eta} u$ in case (a) and $\tilw = w^\mo t_{\eta -\al} u$ in case (b). These are precisely the two shapes called the first form and the second form, respectively, in \cite[Proposition 2.1.2]{LLMPQ-CL} except that $w$ is replaced by $w^\mo$ therein.  
    As in Remark 2.1.3 in \loccit, case (a) occurs when $w_{j_0}^\mo$ preserves the order of $i_0$ and $k_0$ and maps no element $i\in (i_0,k_0)$ into $(w^\mo_{j_0}(i_0),w^\mo_{j_0}(k_0))$. Case (b) occurs when $w^\mo_{j_0}$ reverses the order of $i_0$ and $k_0$ and maps any element $i\in (i_0,k_0)$ into $(w^\mo_{j_0}(k_0),w^\mo_{j_0}(i_0))$.

\item Take $f=1$ for simplicity. By the previous item, $w^\dia$ is special if either (a) $w$ maps $1$ and $n$ to $t_0$ and $t_0+1$ for some $1 \le t_0 \le n-1$ or (b) $w$ interchanges $1$ and $n$. In particular, $w^\dia$ is special if and only if $(w\del)^\dia$ is special for $\del\in S$. Recall from Remark \ref{rmk:def-S} that $(w\delta)^\dia=w^\dia \delta^\dia$. Thus, the definition of special $p$-alcove does not depend on the choice of $w^\dia$ such that $C=w^\dia \cdot C_0$. We can also see that the number of special $p$-alcoves is $(n-2)!$ if $n\ge 3$ and zero if $n\le 2$. For $f\ge 1$, $C=(C_j)_{\jj}$ is special if $C_j$ is special for some $j\in \cJ$.

    \item For $\delta \in \sig_{j_0}(\RG{(12\dots n)})$, $(w^\dia \delta^\dia, \pi^\mo(\delta^\dia)^\mo \cdot (\omega-\eta)+\eta)$ is another lowest alcove presentation of $\sig$. In this way, we can replace $w$ and $u$ by $w\delta$ and $u\delta$, respectively. Then $\tilw$ is replaced by $\delta^{-\dia} \tilw \delta^{\dia}$. For $\tilw$ in case (b), we can choose an appropriate $\delta$,~e.g.~$\del=\sig_{j_0}((12\dots n)^{n-w^\mo_{j_0}(i_0)+1})$, so that $\del^{-\dia} \tilw \del^\dia$ is in case (a). We use this to assume that our $\tilw$ is in case (a). This is not strictly necessary, but it allows us not to repeat the arguments twice when referring \cite{LLMPQ-CL}.
    \end{enumerate}
\end{rmk}

\section{Local model theory}
From now on, we fix $G=\GL_n$. We first review the Emerton--Gee stack and local models. For applications, the key point is that extremal or colength one local charts of potentially crystalline stacks can be described in terms of explicit affine schemes under local model diagrams; this is explained in \S\ref{subsec:LM}. We show in \S\ref{subsub:back-to-LM} that the diagram holds for $(2n-3)$-generic $\tau$. In \S\ref{subsub:special-fiber}, we explain how to compare local charts for potentially crystalline stacks with different $\tau$, which is crucial for our main technical result (Proposition \ref{prop:int}). Everything in \S\ref{sub:review} is an application of the local model theory in \cite{LLLMlocalmodel} to a specific setting, and we do not claim any originality. \S\ref{subsec:chart} is devoted to our main technical result, proved using the explicit local charts studied in \cite{LLMPQ-CL}. Finally, \S\ref{sub:satake} explains the meaning of this result in terms of Hecke operators.

\subsection{Review of the Emerton--Gee stack and local models}\label{sub:review}
We closely follow \cite{LLLMlocalmodel} to quickly review the Emerton--Gee stack and the local model theory used in \S\ref{subsec:chart}.

\subsubsection{The Emerton--Gee stack}
Let $\cX_n$ be the moduli stack of rank $n$ projective \'etale $\pgma$-modules constructed in \cite{EGstack}. It is a Noetherian formal algebraic stack over $\Spf \cO$. For a tame inertial type $\tau$, there exists a closed $p$-adic formal substack $\cX^{\eta,\tau}_n\subset \cX_n$ whose $R$-points for any finite flat $\cO$-algebra $R$ parameterize lattices in potentially crystalline representations of Hodge type $\eta$ and inertial type $\tau$ over $R$. Let $\cX_{n,\red}$ be the reduced substack of $\cX_n$. It is a finite type algebraic stack over $\F$ equidimensional of dimension $\frac{n(n-1)f}{2}$. Its irreducible components are naturally labeled by Serre weights of $G(k)$ \cite[Theorem 6.5.1]{EGstack}. We denote by $\cC_\sig$ the component corresponding to a Serre weight $\sig$ following the normalization in \cite[\S7.4]{LLLMlocalmodel}. If we write $\sig=F(\lam)$ for some $\lam \in X^*(T)_1^\cJ$, then $\cC_{\sig}$ is characterized as the closure of the locus of Galois representations $\rhobar$ of the form
\begin{align}\label{eqn:ord-rhobar}
    \rhobar = \pma{\chi_1 & * & \cdots & * \\ 
     0 & \chi_2 & \cdots & * \\ 
     \text{\rotatebox{90}{$\cdots$}} &  & \text{\rotatebox{135}{$\cdots$}} &  \text{\rotatebox{90}{$\cdots$}} \\ 
     0 & \cdots & 0 & \chi_n}
\end{align}
where
\begin{enumerate}
    \item $\rhobar$ is maximally non-split of niveau $1$, i.e.~it admits a unique $G_K$-stable flag;
    \item $\rhobar^\ss|_{I_K} = \oplus_{i=1}^{n}\chi_i|_{I_K} = \prod_{\jj}\oom_{K,\sig_j}^{(\lam+\eta)_j}$ as $T^\vee(\F)$-valued representations;
    \item If $\chi_i\chi_{i+1}^\mo |_{I_K} = \ov{\veps}$, then $\lam_{j,i} - \lam_{j,i+1} = p-1$ for all $\jj$ if and only if  $\chi_i\chi_{i+1}^\mo= \ov{\veps}$ and the element of $\Ext^1_{G_K}(\chi_i,\chi_{i+1})$ determined by $\rhobar$ is tr\`es ramifi\'ee. Otherwise, $\lam_{j,i} - \lam_{j,i+1}  = 0$ for all $\jj$.
\end{enumerate}

\subsubsection{Local model diagram}\label{subsec:LM}
In \cite{LLLMlocalmodel}, the authors developed a local model theory describing open charts of $\cX^{\eta,\tau}_n$ explicitly for sufficiently generic $\tau$. They define an open cover $\{\cX^{\eta,\tau}_{n}(\tilz)\}_{\tilz}$ of $\cX^{\eta,\tau}_n$ labeled by $\tilz\in \Adm^\vee(\eta)^\cJ$. For each $\tilz$, they define an explicit affine scheme $U^{\eta,\nblat}(\tilz)$ over $\cO$ with a local model diagram
\begin{equation}\label{eqn:LM}
\begin{tikzcd}
    & T^{\vee,\cJ} \times_{\Spec \cO} U^{\eta,\nblat}(\tilz)^\pcp \arrow[ld] \arrow[rd] \\
    \cX^{\eta,\tau}_n(\tilz) & & U^{\eta,\nblat}(\tilz)^\pcp.
\end{tikzcd}
\end{equation}
Here, $\pcp$ denotes the $p$-adic completion and both diagonal arrows are $T^{\vee,\cJ}$-torsors \cite[Theorem 7.3.2]{LLLMlocalmodel}. In \loccit, an inexplicit genericity condition on $\tau$ is required. As we are only interested in $\tilz$ either extremal or colength one, we explain why the above diagram holds for $(2n-3)$-generic $\tau$.  

We give a definition of $U^{\eta,\nblat}(\tilz)$ in \S\ref{subsub:back-to-LM}. For now, we mention that it is an affine scheme inside $\cJ$-fold product of the loop group functor $L\cG$: for a Noetherian $\cO$-algebra $R$, 
\begin{align*}
    L\cG : R \mapsto L\cG(R):=\{ A\in \GL_n(R\DP{v+p}) \mid A \mod v \text{ is upper triangular}\}.
\end{align*}

When $\tau$ is $(n+1)$-generic, there is a closed immersion of $\cX^{\eta,\tau}_n$ into the moduli of rank $n$ Breuil--Kisin modules $Y^{\le\eta,\tau}$ of height $\le \eta$ with tame descent data $\tau$, and its image is identified with the locus of Breuil--Kisin modules of height $\eta$ satisfying the monodromy condition \cite[Proposition 7.2.3]{LLLMlocalmodel}. There is an open substack $Y^{\le\eta,\tau}(\tilz)\subset Y^{\le\eta,\tau}$ of Breuil--Kisin modules admitting $\tilz$-gauge basis. Then $\cX^{\eta,\tau}_n(\tilz)$ is defined by the inverse image of $Y^{\le\eta,\tau}(\tilz)$ in $\cX^{\eta,\tau}_n$.

For $(n+1)$-generic $\tau$, there is a local model diagram for $Y^{\le\eta,\tau}$ (Theorem 5.3.3 in \loccit):
\begin{equation}\label{eqn:LM-BK}
\begin{tikzcd}
    & Y^{\le\eta,\tau,\be}(\tilz) \simeq T^{\vee,\cJ}\times {U}^{\le\eta}(\tilz)^\pcp \arrow[rd] \arrow[ld] & \\
    Y^{\le\eta,\tau}(\tilz) & & U^{\le\eta}(\tilz)^\pcp
\end{tikzcd}
\end{equation}
where $Y^{\le\eta,\tau,\be}(\tilz)$ parameterizes Breuil--Kisin modules $\fM$ in $Y^{\le\eta,\tau}(\tilz)$ together with a gauge basis $\be$. The left diagonal arrow is forgetting $\be$ and is a $T^{\vee,\cJ}$-torsor. The scheme ${U}^{\le\eta}(\tilz)$ is also contained in $L\cG^\cJ$, and the top isomorphism is given by taking $(\fM,\be)$ to its tuple of Frobenius matrices $(A_{\fM,\be}\ix{j})_{\jj}$. The right diagonal arrow is given by projection.

There is a $p$-adic formal subscheme ${U}^{\eta,\nbl_\infty}(\tilz)\subset {U}^{\le\eta}(\tilz)^\pcp$ such that $T^{\vee,\cJ}\times {U}^{\eta,\nbl_\infty}(\tilz)$ is the pullback of $\cX^{\eta,\tau}(\tilz)$ along the left diagonal arrow in \eqref{eqn:LM-BK}. Then \eqref{eqn:LM} is induced by an isomorphism between ${U}^{\eta,\nbl_\infty}(\tilz)$ and ${U}^{\eta,\nblat}(\tilz)^\pcp$ obtained by applying Elkik's approximation theorem which requires an inexplicit genericity. This also requires the fact that ${U}^{\eta,\nblat}(\tilz)$ is irreducible, proven in \cite[Theorem 3.7.1]{LLLMlocalmodel} under inexplicit genericity. When $\tilz$ is extremal or colength one, ${U}^{\eta,\nblat}(\tilz)$ can be computed directly and is irreducible. We can also perform Elkik's approximation by hand in this case.

\subsubsection{Explicit local charts}\label{subsub:chart}
Recall the functor $L\cG$ defined above. We introduce various subfunctors of $L\cG$: for a Noetherian $\cO$-algebra $R$,
\begin{enumerate}
    
    \item $L^+\cG: R \mapsto \{A\in G(R\DB{v+p}) \mid A \mod v \in B(R)\}$ (the positive loop group);
    
    \item $L^{[0,n-1]}\cG : R \mapsto \{A \in L\cG(R) \mid A, v^{n-1} A^\mo \in \Lie G(R\DB{v+p})\}$;
    \item For $\bfa \in \cO^n$, which we view as an $n\times n$ diagonal matrix,
    \begin{align*}
        L\cG^{\nbla}: R \mapsto \{A \in L\cG(R) \mid (v+p)(v\frac{dA}{dv} A^\mo + A \bfa A^\mo) \in \Lie G(R\DB{v+p}) \text{ and is upper triangular mod $v$}\};
    \end{align*}
\end{enumerate}
If $X\subset L\cG$ is a subfunctor, we write $X^{\nbla}:= X \cap L\cG^{\nbla}$. We put a subscript $E$ or $\F$ to denote the base change of these functors to $E$ or $\F$.

We define $\Gr_\cG$ to be the fpqc quotient $[L^+\cG \backslash L\cG]$ and $\Gr_\cG^{[0,n-1]}:= [L^+\cG \backslash L\cG^{[0,n-1]}]$. The latter is a projective scheme over $\cO$. Let $M^{\le\eta}$ be the Pappas--Zhu local model associated to the group $\GL_n$, the conjugacy class of $\eta$, and the Iwahori subgroup \cite{PZ13-Inv-local_model-MR3103258}. It is a projective subscheme of   $\Gr_{\cG}^{[0,n-1]}$. For $\tilz \in \Adm^\vee(\eta)$, $U^{\le\eta}(\tilz)\subset L\cG$ is an affine scheme and it defines an open subscheme of $M^{\le\eta}$ under the natural map $L\cG \ra \Gr_{\cG}$. Then $U^{\le\eta,\nbla}(\tilz)$ is defined by the $\cO$-flat part of $U^{\le\eta}(\tilz)^{\nbla}$. For the cases of our interests, we have the following explicit description.

\begin{prop}\label{prop:chart}
    Suppose that $\bfa$ is $n$-generic (as defined in \cite[\S4.2]{LLLMlocalmodel}).
    \begin{enumerate}
        \item If $\tilz$ is extremal, then $U^{\le\eta,\nbla}(\tilz)$ is an affine space over $\cO$ of relative dimension $n(n-1)/2$.
        \item If $\tilz$ is of colength one, then $U^{\le\eta,
    \nbla}(\tilz)$ is an affine space over $\cO[c,Z_{-\al}]/(c Z_{-\al}-p)$ of relative dimension $n(n-1)/2 -1$.
    \end{enumerate}
    In particular, $U^{\le\eta,\nbla}(\tilz)$ is normal in both cases.
\end{prop}

Note that the functions $c,Z_{-\al} \in \cO(U^{\le\eta,\nbl_{\bfa}}(\tilz))$ are explicitly defined in \cite[Proposition 4.3.2]{LLMPQ-CL}.

\begin{proof}
See \cite[Proposition 4.3.1 and 4.3.2]{LLMPQ-CL}. Note that the closed embeddings in \loccit~are isomorphisms because the target is normal and the source and the target have the same dimension.
\end{proof}

\begin{rmk}
    When $p>n-1$, for any $\bfa\in \cO^n$ and $\tilz\in \Adm^\vee(\eta)$, $U^{\le\eta,\nbla}(\tilz)\times_{\Spec \cO} \Spec E$ contains a unique connected component of dimension $n(n-1)/2$, given by its intersection with the Schubert cell
    \begin{align*}
        S^\circ_E(\eta)= L^+\cG_E \backslash L^+\cG_E (v+p)^\eta L^+\cG_E
    \end{align*}
    (see \cite[Proposition 3.3.4]{LLLMlocalmodel}). Then $U^{\eta,\nbla}(\tilz)$ is defined by the closure of such a component in $U^{\le\eta,\nbla}(\tilz)$. In the context of Proposition \ref{prop:chart}, $U^{\le\eta,\nbla}(\tilz)\times_{\Spec \cO} \Spec E$ is already connected and thus $U^{\eta,\nbla}(\tilz) = U^{\le \eta,\nbla}(\tilz)$. 
\end{rmk}

\subsubsection{Back to the local model diagram}\label{subsub:back-to-LM}
Let $\tau=\tau(s,\mu+\eta)$ be a tame inertial type. We define $\bfa_\tau=(\bfa_{\tau,j})_{\jj}$ as in \cite[\S7.3]{LLLMlocalmodel}. Then $\bfa_{\tau,j}\in \cO^n$ and $\bfa_{\tau,j} \equiv s_j^\mo(\mu_j+\eta_j) \mod \varpi \in \F_p^n$. 
For $\tilz= (\tilz_j)_{\jj}\in \Adm^\vee(\eta)^\cJ$, we define
\begin{align*}
    U^{\eta,\nblat}(\tilz) := \prod_{\jj}U^{\eta,\nbl_{\bfa_{\tau,j}}}(\tilz_j)\subset M^{\le\eta,\cJ}.
\end{align*}
Recall $U^{\eta,\nbl_\infty}(\tilz)$ from \S\ref{subsec:LM}. If $\tau$ is $(2n-3)$-generic, then Proposition 7.1.10 in \loccit~implies that
\begin{align}\label{eqn:modp2}
    U^{\eta,\nbl_\infty}(\tilz)\times_{\Spf \cO}\Spec \cO/p^{2} \subset U^{\eta,\nblat}(\tilz).
\end{align}
Furthermore, assume that $\tilz_j$ is either extremal or of colength one for all $\jj$. Using $\cO$-flatness of $U^{\eta,\nbl_\infty}(\tilz)$ and the explicit description of $U^{\eta,\nblat}(\tilz)$ (Proposition \ref{prop:chart}), we can lift the above inclusion to a morphism
\begin{align*}
    U^{\eta,\nbl_\infty}(\tilz) \mono U^{\eta,\nblat}(\tilz)^\pcp
\end{align*}
which is a closed immersion by $p$-completeness of the source and the target. To be precise, we need to construct a ring map 
\begin{align*}
    \cO(U^{\eta,\nblat}(\tilz)^\pcp) \ra \cO(U^{\eta,\nbl_\infty}(\tilz))
\end{align*}
lifting the ring map induced by \eqref{eqn:modp2}. We can send free variables in $\cO(U^{\eta,\nblat}(\tilz))$ to any lift of their images under \eqref{eqn:modp2}. For $c$ and $Z_{-\al}$, we can find lifts of their images, denoted by $x$ and $y$, satisfying $xy=p+p^2z$ for some $z\in \cO(U^{\eta,\nbl_\infty}(\tilz))$. Since $1+pz$ is invertible, we can map $c$ and $Z_{-\al}$ to $x$ and $y/(1+pz)$. Then the constructed morphism is an isomorphism because its source and target have the same dimension and the target is normal. This justifies the existence of the diagram \eqref{eqn:LM} for $\tilz$ of colength at most one.

\subsubsection{Special fiber of potentially crystalline stacks}\label{subsub:special-fiber} 
Since $\cX^{\eta,\tau}_{n,\F}$ is an algebraic stack over $\F$ of dimension $n(n-1)f/2$, it is topologically a union of irreducible components $\cC_\sig$. If we assume $\tau$ to be $(5n-1)$-generic, these irreducible components are precisely $\cC_\sig$ for $\sig\in \JH(\osig(\tau))$ \cite[Theorem 7.4.2]{LLLMlocalmodel}. We only need weaker results that hold under $(2n-3)$-genericity of $\tau$ (see Lemma \ref{lem:chart-extremal} and \ref{lem:chart-CLone}). We record some preliminary results here.

If we apply the base change along $\cO\epi \F$, $L\cG$ becomes the loop group $LG_\F$, $L^+\cG$ becomes the Iwahori subgroup $\cI_\F$, and $\Gr_\cG$ becomes the affine flag variety $\Fl:=[\cI_\F\bss LG_\F]$. Let $\cI_{1,\F}$ be the pro-$v$ Iwahori subgroup. Then $\tld{\Fl}:=[\cI_{1,\F}\bss LG_\F]$ is a $T^{\vee}$-torsor over $\Fl$. For $\bfa\in \cO^n$, we have $\Fl^{\nbla}:= [\cI_\F\bss LG_\F^{\nbla}]$ and $\tld{\Fl}^{\nbla}:= [\cI_{1,\F}\bss LG_\F^{\nbla}]$. If $X\subset \Fl$ is a subscheme, we write $X^{\nbla} = X\cap \Fl^{\nbla}$ and write $\tld{X}$ for the pullback of $X$ along $\tld{\Fl} \ra \Fl$. Note that $M^{\le\eta}_\F$ is contained in $\Fl$.

For $\tilz=t_\nu w\in \tilW^\vee$, we have the Schubert cell $S^\circ_\F(\tilz) := \cI_\F\backslash \cI_\F v^\nu w \cI_\F$, and we define $S^{\nbla}_\F(\tilz)$ to be the closure of $S^\circ_\F(\tilz)^{\nbla}$ in $\Fl^{\nbla}$. 

\begin{prop}[{\cite[Corollary 4.2.6]{LLLMlocalmodel}}]\label{prop:irr(M)}
    Suppose that $\tau$ is $(n-1)$-generic. Then there is a natural bijection between $\Adm^{\reg,\vee}(\eta)^\cJ$ and the set of top-dimensional irreducible components of $M^{\le\eta,\nblat}_\F$ given by $\tilz\mapsto S^{\nblat}_\F(\tilz)$. 
\end{prop}

\begin{rmk}[{cf.~\cite[Remark 5.2.5]{LLLMlocalmodel}}]\label{rmk:S-in-U}
    If $\tilz \in \Adm^\vee(\eta)$, we have
\begin{align*}
    S_\F^\circ(\tilz)^{\nbla} \subseteq U_\F^{\le\eta,\nbla}(\tilz)
\end{align*}
where we view $U_\F^{\le\eta,\nbla}(\tilz)$ as a subscheme in $\Fl$ using the natural map. When $\bfa$ is $n$-generic and $\tilz$ is extremal, both of them are affine spaces of dimension $n(n-1)/2$ (Theorem 4.2.4 in \loccit~and Proposition \ref{prop:chart}), and thus the equality holds.
\end{rmk}

Suppose that $\tau$ is $(2n-4)$-generic. 
The image of $\cX^{\eta,\tau}_{n,\F}$ in $Y^{\le\eta,\tau}$ is smoothly equivalent to $M^{\le\eta,\nblat}_\F$ under the local model diagram by \cite[Proposition 7.4.1]{LLLMlocalmodel}. Thus, Proposition \ref{prop:irr(M)} gives an upper bound of the set of irreducible components in $\cX^{\eta,\tau}_{n,\F}$. In a special case, we can show that a certain top-dimensional component of $M^{\le\eta,\nblat}_\F$ does give rise to an irreducible component in $\cX^{\eta,\tau}_{n,\F}$. We also explain how to compare local models for $\cX^{\eta,\tau}_{n,\F}$ for different choices of $\tau$.

For a triple $(\tilw_1,\tilw_2,\tils)\in (\tilW^+)^2\times \tilW$, we define $S^{\nblz}_\F(\tilw_1,\tilw_2,\tils)$ to be the closure of $S^\circ_\F((\tilw_2^\mo w_0 \tilw_1)^*)\tils^* \cap \Fl^{\nblz}$ in $\Fl^{\nblz}$. If $\tils\in \tilW$ such that $\tils^*(0) \equiv \bfa \mod p$ and $\tilz = (\tilw_2^\mo w_0 \tilw_1)^*$, then the right multiplication by $\tils$, denoted by $r_{\tils^*}$, embeds $M_\F^{\le\eta,\nbla}$ into $\Fl^{\nblz}$ and maps $S^{\nbla}_\F(\tilz)$ to $S^{\nblz}_\F(\tilw_1,\tilw_2,\tils)$ isomorphically \cite[Proposition 4.3.1]{LLLMlocalmodel}. If $\tils^*(0)$ is $(n-1)$-generic and $\tilw_1,\tilw_2\in\tilW_1$, then for any $w\in W$, we have (see Proposition 4.3.5 and 4.3.6 in \loccit)
\begin{align}\label{eqn:schubert-modification}
    S^{\nblz}_\F(\tilw_1,\tilw_2,\tils) = S^{\nblz}_\F(\tilw_1,e,\tils \tilw_2^\mo w).
\end{align}

For $(2n-4)$-generic $\tau$, we have a commutative diagram (cf.~\cite[Theorem 7.4.1]{LLLMlocalmodel} and the preceding discussion)
\begin{equation}\label{eqn:LM-modp}
    \begin{tikzcd}
    T^{\vee,\cJ}\times U_\F^{\eta,\nblat}(\tilz) \arrow[r, "r_{\tilw^*(\tau)}", hook] \arrow[d] & \tld{M}^{\le\eta,\nblat}_\F \tilw^*(\tau)  \arrow[d] \arrow[r, symbol=\subset] & \tld{\Fl}^{\nblz,\cJ} \arrow[d] \\
    \cX^{\eta,\tau}_{n,\F}(\tilz) \arrow[r, hook] & \Phi\mathrm{\dash Mod}^{\et,n}_{K,\F} & \Fl^{\nblz,\cJ}.
\end{tikzcd}
\end{equation}
Here, $\Phi\mathrm{\dash Mod}^{\et,n}_{K,\F}$ is the moduli stack of \'etale $\varphi$-modules and the bottom horizontal map is given by the restriction of $G_K$-representations to $G_{K_\infty}$ where $K_\infty/K$ is a totally ramified infinite extension obtained by $p$-power roots of $-p$. 
The upshot is that we can compare local models of $\cX_{n,\F}^{\eta,\tau}(\tilz)$ for different choices of $\tau$ and $\tilz$ in $\tld{\Fl}^{\nblz,\cJ}$. Moreover, we can match irreducible components in $\cX^{\eta,\tau}_{n,\F}(\tilz)$ with 
some of $S^{\nblz}_\F(\tilw_1,\tilw_2,\tils)$ in $\Fl^{\nblz,\cJ}$.

Suppose that $\sig$ is $(3n-4)$-deep with lowest alcove presentation $(w^\dia,\om)$. For $s\in W^\cJ$, we define
\begin{align*}
    \tau_s := \tau(\pi^\mo(w)^\mo s w , \om +\pi^\mo(w)^\mo(\nu_w-\eta)),
\end{align*}
which is $(2n-3)$-generic. Then $\sig \in \JH(\osig(\tau_s))$ by \cite[Theorem 1.1]{LLLM-DL}.

\begin{lem}\label{lem:component}
    With the above notation, $\cC_\sig$ is contained in $\cX^{\eta,\tau_s}_{n,\F}$. Moreover, it corresponds to $S^{\nblz}_\F(w^\dia,e,t_\om)$ in $\Fl^{\nbl_0,\cJ}$ under the diagram \eqref{eqn:LM-modp}.  
\end{lem}
\begin{proof}
    By Remark \ref{rmk:S-in-U} and the local model diagram \eqref{eqn:LM}, $S^{\nblat}(\tilz)$ for 
    \begin{align*}
        \tilz = w^{-\dia}\tilw_h^\mo w_0 w^{\dia} = t_{w^\mo(\eta)}
    \end{align*}
    corresponds to an irreducible component in $\cX^{\eta,\tau_s}_{n,\F}$. To see that the corresponding component is $\cC_\sig$, we can argue as in the proof of \cite[Lemma 7.4.6]{LLLMlocalmodel} (this is done for $s=e$, but the same computation for general $s$ proves our claim). Finally, $S^{\nblat}(\tilz)$ is matched with $S^{\nblz}_\F(w^\dia,e,t_\om)$ by \eqref{eqn:schubert-modification}.
\end{proof}

\subsection{Intersection of local charts}\label{subsec:chart}
In this subsection, we specialize to the following setup.
\begin{setup}\label{setup}
    Recall from \S\ref{subsec:sp-alc} that we have $\sig=F_{(w^\dia,\om)}$ and $\sig'=F_{(u^\dia,\om)}$ with $w=s_\al u$ and $\al=\sig_{j_0}(\al_{i_0k_0})$. We can and do choose $w$ so that $\nu_w=\nu_u$ and $\tilw(\rhobar,\tau)=w^\mo t_\eta u$ as explained in Remark \ref{rmk:special}(3). 
\begin{itemize}
    \item $\rhobar\simeq \tau(\pi^\mo(w)^\mo u,\om+\pi^\mo(w)^\mo(\nu_u))$ a tame inertial type over $\F$
    \item $\tau\simeq \tau(\pi^\mo(w)^\mo w,\om+\pi^\mo(w)^\mo(\nu_w-\eta))$ a tame inertial type over $E$
    \item $\tilz := \tilw^*(\rhobar,\tau)$ where $\tilw(\rhobar,\tau) = w^\mo t_{\eta+\nu_u-\nu_w} u$
    \item $\rhobar'\simeq \tau(\pi^\mo(u)^\mo u, \om+\pi^\mo(u)^\mo(\nu_u))$ a tame inertial type over $\F$
    \item $\tau'\simeq \tau(\pi^\mo(u)^\mo u,\om+\pi^\mo(u)^\mo(\nu_u-\eta))$ a tame inertial type over $E$
    \item $\tilz':= \tilw^*(\rhobar',\tau')$ where $ \tilw(\rhobar',\tau')= t_{u^\mo(\eta)}$.
\end{itemize}
    We assume that $\sig$ is $(3n-4)$-deep. This implies that $\tau$ and $\tau'$ are $(2n-3)$-generic. 
\end{setup}

\begin{rmk}
    Let us explain the objects in Setup \ref{setup}.  Using \eqref{change-of-LAP}, one can check that
\begin{align*}
    {\pi^\mo(w^\dia)}\cdot (\pi^\mo(w)^\mo w,\om+\pi^\mo(w)^\mo(\nu_w-\eta)-\eta) = (e, \lambda-\eta)
\end{align*}
which implies that $\sig(\tau)\simeq \Ind_{B(k)}^{G(k)} [\tld{\lam}]$ where $\tld{\lam}=\sig^{U(k)}$ as in Example \ref{ex:PS}. Such a choice is useful because of Lemma \ref{lem:upperbound} (note that $\JH(\Ind_{B(k)}^{G(k)} \tld{\lam}) = \JH(\Ind_{\ov{B}(k)}^{G(k)} \tld{\lam})$). Then $\rhobar$ is chosen so that $\tilw(\rhobar,\tau) = (\tilw(\rhobar,\tau)_j)_{\jj}$ is the shape of colength one at $j=j_0$ and extremal at $j\neq j_0$ ``capturing $\sig$ and $\sig'$''. This means that $\cX^{\eta,\tau}_{n,\F}(\tilz)$ has two irreducible components that are open dense substacks of $\cC_{\sig}$ and $\cC_{\sig'}$. We denote these irreducible components by $\cC^\tau_{\sig}(\tilz)$ and $\cC^\tau_{\sig'}(\tilz)$, respectively.  

Similarly, we have $\rhobar'\simeq \prod_{\jj}\oom_{K,\sig_j}^{(\lam'+\eta)_j}$ and $\sig(\tau')\simeq \Ind_{B(k)}^{G(k)} [\tld{\lam}']$ where $\tld{\lam}'=(\sig')^{U(k)}$. Then $\tau'$ is chosen so that $\tilw(\rhobar',\tau')=t_{u^\mo(\eta)}$ is an extremal shape. In this case, $\cX_{n,\F}^{\eta,\tau'}(\tilz')$ is irreducible and is the ``ordinary locus'' of $\cC_{\sig'}$. We write $\cC^{\tau'}_{\sig'}(\tilz')=\cX^{\eta,\tau'}_{n,\F}(\tilz')$.
\end{rmk}

The following is the main result of this subsection.

\begin{prop}\label{prop:int}
    Following the above notations, we have
    \begin{align*}
        \cC^\tau_\sig(\tilz)\cap \cC^\tau_{\sig'}(\tilz) \cap \cC^{\tau'}_{\sig'}(\tilz') \neq \emptyset.
    \end{align*}
\end{prop}

The intersection $\cC^\tau_\sig(\tilz)\cap \cC^\tau_{\sig'}(\tilz)$ is a closed substack in $\cC^\tau_{\sig'}(\tilz)$. On the other hand, $\cC^\tau_{\sig'}(\tilz)\cap \cC^{\tau'}_{\sig'}(\tilz')$ is a non-empty open in $\cC^\tau_{\sig'}(\tilz)$ as it is the intersection of two open substacks in  $\cC^\tau_{\sig'}(\tilz)$. Thus, the above proposition asserts that these closed and open loci are not disjoint in $\cC^\tau_{\sig'}(\tilz)$.

We first justify the claims made above about $\cX_{n,\F}^{\eta,\tau}(\tilz)$ and $\cX_{n,\F}^{\eta,\tau'}(\tilz')$.

\begin{lem}\label{lem:chart-extremal}
   The algebraic stack $\cC^{\tau'}_{\sig'}(\tilz')=\cX^{\eta,\tau'}_{n,\F}(\tilz')$ is an open dense substack of $\cC_{\sig'}$ whose $\F$-points consist of $\rhobar$ admitting a complete flag with $\gr(\rhobar)\simeq \prod_{\jj}\oom_{K,\sig_j}^{(\lam'+\eta)_j}$ as $T^\vee(\F)$-valued representations. 
\end{lem}
\begin{proof}
    This follows from \cite[Lemma 7.4.6]{LLLMlocalmodel}. 
\end{proof}

\begin{lem}\label{lem:chart-CLone}
    The algebraic stack $\cX^{\eta,\tau}_{n,\F}(\tilz)$ has two irreducible components $\cC^\tau_{\sig'}(\tilz)$ and $\cC^\tau_{\sig}(\tilz)$, which are open dense substacks in $\cC_{\sig'}$ and $\cC_{\sig}$, respectively.
\end{lem}
\begin{proof}
    Consider $U^{\eta,\nblat}_\F(\tilz)\tilw^*(\tau)\subset \Fl^{\nblz,\cJ}$ as in \eqref{eqn:LM-modp}. We show that its two irreducible components are open dense in $S^{\nblz}_\F(w^\dia,e,t_\om)$ and $S^{\nblz}_\F(u^\dia,e,t_\om)$. Then the claim follows from Lemma \ref{lem:component}. 

    Recall that $U_\F^{\eta,\nblat}(\tilz)$ is contained in $(M^{\le\eta,\cJ}_\F)^{\nblat}$, whose top-dimensional irreducible components are of the form $S_\F^{\nblat}(\tils)$ for $\tils\in \Adm^{\reg,\vee}(\eta)^\cJ$ (Proposition \ref{prop:irr(M)}). Since
    \begin{align*}
        S_\F^{\nblat}(\tils)\subset S_\F(\tils) = \cup_{\tils'\le \tils} S^\circ_\F(\tils'),
    \end{align*}
    $S_\F^{\nblat}(\tils)$ contains $\tilz$ only if $\tils=\tilz = (w^{-\dia} \tilw_h^\mo w_0 u^\dia)^*$ or $(w^{-\dia} \tilw_h^\mo w_0 w^\dia)^*$. Since both irreducible components of $U_\F^{\eta,\nblat}(\tilz)$ contain $\tilz$, they are open subschemes of $S_\F^{\nblat}(\tils)$ for such two $\tils$. Using \eqref{eqn:schubert-modification}, we have
    \begin{align*}
        S^{\nblat}_\F((w^{-\dia} \tilw_h^\mo w_0 s^\dia)^*)\tilw^*(\tau)  &= S^{\nblz}_\F(s^\dia, \tilw_h w^\dia, \tilw(\tau)) \\ &= S^{\nblz}_\F(s^\dia, e, \tilw(\tau) w^{-\dia} \tilw_h^\mo) \\&= S^\nblz_\F(s^\dia,e, t_\om)
    \end{align*}
    for $s\in \{u,w\}$. 
\end{proof}

The remainder of this subsection is devoted to the proof of Proposition \ref{prop:int}. Since the local model diagram is given by $T^{\vee,\cJ}$-torsors, non-emptiness of the intersection can be checked after pullback to the corresponding local model chart. 
Recall from Proposition \ref{prop:chart} that we have explicit functions $c$ and $Z_{-\al}$ on the $j_0$th component of $U^{\eta,\nblat}_\F(\tilz)$. We write $V(c)$ and $V(Z_{-\al})$ for the vanishing loci of $c$ and $Z_{-\al}$ in $U^{\eta,\nblat}_\F(\tilz)$. 

\begin{lem}\label{lem:int}
    The intersection $(V(c) \cap V(Z_{-\al})) \tilw^*(\tau) \cap S_\F^\circ(\tilz')^{\nbl_{\bfa_{\tau'}}}\tilw^*(\tau')$ is non-empty.
\end{lem}

We first describe $V(c)$ in terms of a Schubert cell. By \eqref{eqn:schubert-modification} and \cite[Proposition 4.3.1]{LLLMlocalmodel}, we have
\begin{align*}
    S^\circ_\F(\tilz)^{\nblat} = (S^\circ_\F((w_0u^\dia)^*) \tilw_h^{-*} w^{-\dia*})^{\nblat}.
\end{align*}
There is an affine scheme $N_{(w_0u^\dia)^*}\subset LG_\F^{\cJ}$ such that the natural map $(w_0u^\dia)^* N_{(w_0u^\dia)^*}\ra S^\circ_\F((w_0u^\dia)^*)$ is an isomorphism \cite[Proposition 4.2.13]{LLLMlocalmodel}. We denote by $N_{(w_0u^\dia)^*}^{\nblat'} \subset N_{(w_0u^\dia)^*}$ the closed subscheme given by the condition
\begin{align*}
    (w_0u^\dia)^*N_{(w_0u^\dia)^*}^{\nblat'} \tilw_h^{-*}w^{-\dia*} = ((w_0u^\dia)^*N_{(w_0u^\dia)^*} \tilw_h^{-*}w^{-\dia*})^{\nblat}.
\end{align*}

\begin{lem}\label{lem:V(c)}
    We have $V(c) = S^\circ_\F(\tilz)^{\nblat}$ as subvarieties of $\Fl^{\nblat}$. Moreover, $V(c) = (w_0u^\dia)^*N_{(w_0u^\dia)^*}^{\nblat'}\tilw_h^{-*} w^{-\dia*}$ as subvarieties of $LG_\F^{\cJ}$.
\end{lem}

\begin{proof}
    The first claim follows from the second. We need to recall the explicit description of $V(c)$ from \cite{LLMPQ-CL}. They first compute $U_\F^{\le\eta}(\tilz)$ (Proposition 3.3.1 in \loccit) and then compute $U_\F^{\eta,\nblat}(\tilz)$ by adding monodromy conditions (Proposition 4.3.2). The coordinate $c$ is already defined for $U_\F^{\le\eta}(\tilz)$, and we denote its vanishing locus by $V'(c)$ to ease the notation. We show that $V'(c)$ contains $(w_0u^\dia)^*N_{(w_0u^\dia)^*}\tilw_h^{-*} w^{-\dia*}$ (as subvarieties of $LG_\F$). By taking intersection with $\Fl^{\nblat}$, this implies that $V(c)$ contains $(w_0u^\dia)^*N_{(w_0u^\dia)^*}^{\nblat'}\tilw_h^{-*} w^{-\dia*}$. Since they are affine spaces of the same dimension, this implies the second claim.

    Let $R:= \cO(U_\F^{\le\eta}(\tilz))$.  Proposition 3.1.2 in \loccit~shows that $V'(c)$ is given by the universal matrix $A\in LG_\F(R)$ such that
    \begin{align*}
        A=u^\mo v^\eta V w
    \end{align*}
    for $V\in \ov{U}(R[v])$ with entries satisfying certain degree bound. Comparing this with $(w_0u^\dia)^*N_{(w_0u^\dia)^*}\tilw_h^{-*} w^{-\dia*}$, we want to show that
    \begin{align*}
         \Ad_{v^{-\eta+\nu_w} w_0}(N_{(w_0 u^\dia)^*}) \subset V.
    \end{align*}
     Note that $N_{(w_0 u^\dia)^*}\in U(R[v])$. Thus, $\Ad(v^{-\eta+\nu_w} w_0)(N_{(w_0 u^\dia)^*})\in \ov{U}(R[v])$. We need to show that the degree of entries of $\Ad_{v^{-\eta+\nu_w} w_0}(N_{(w_0 u^\dia)^*})$ matches those of $V$.

    For $\be \in \Phi^-$, the degree of $\be$-entry of $\Ad(v^{-\eta+\nu_w} w_0)(N_{(w_0 u^\dia)^*})$ is given by $\RG{-\eta+\nu_w,\be^\vee} + \deg N_{(w_0u^\dia)^*,w_0(\be)}$ and by \cite[Corollary 4.2.12]{LLLMlocalmodel}
    \begin{align*}
        \deg N_{(w_0 u^\dia)^*,w_0(\be)} = \floor{\RG{u^\dia(x),-\be^\vee}} = \RG{-\nu_u,\be^\vee} - \del_{u^\mo (\be) > 0}.
    \end{align*}
    On the other hand, $\deg V_{\be}$ is given by $\kappa_\beta + m'_\beta$ defined in \cite[\S3]{LLMPQ-CL}. We have $\RG{-\eta,\be^\vee}-2 \le m'_\beta \le \RG{-\eta,\be^\vee}$ and define $\sig_\be := m'_\be + \RG{\eta,\be^\vee}$. Then we need to check that
    \begin{align*}
        \kappa_\be +\sig_\be = - \del_{u^\mo (\be) > 0}.
    \end{align*}
    This follows by unwrapping the definitions of $\ka_\be$ and $\sig_\be$.

    In what follows, we freely use technical terms from \cite[\S3]{LLMPQ-CL}. Note that our $w$, $u$, $k_0$ are $w^\mo$, $(w')^\mo$, $j_0$ in \loccit\  We treat the three possible cases of $\sig_\be$ as in \S3.3 equation (10), subdivided by cases in Proposition 3.1.2 in \loccit
    \begin{itemize}
        \item Suppose that $\be$ is bad. This is case (1) in Proposition 3.1.2. We get $\ka_\be=\sig_\be=\del_{u^\mo (\be) > 0}=0$. 
         \item Suppose that $\be$ shares the column of $-\al$ and $s_\al(\be)\in \Phi^-$ bad. This is part of case (2) in Proposition 3.1.2. Then $\sig_\be=-2$. In this case, $s(\be)\in \Phi^-$ implies $\be = \al_{ji_0}$ for $j> k_0$. Since $s(\be)$ is bad, we have $\del_{u^\mo(\be)<0}\neq \del_{w^\mo(\be)}<0$. This is only possible when $w^\mo(j)$ is in $(w^\mo(i_0),w^\mo(k_0))$, in which case $\ka_\be=1$ and $\sig_\be=-1$.
    \item Suppose otherwise. This is (part of) cases (2,3,4,5) in Proposition 3.1.2. In this case, $\sig_\be=-1$. Using
    \begin{align*}
        -\del_{u^\mo(\be)>0} = \del_{u^\mo(\be)<0} -1,
    \end{align*}
    we need to check $\ka_\be = \del_{u^\mo(\be)<0}$. In case (3), this is true. In the remaining cases, we need to show that $\del_{w^\mo(\be)<0} = \del_{u^\mo(\be)<0}$. When $\be$ does not share row or column of $\al$ or $-\al$, $u^\mo(\be)=w^\mo(\be)$. Otherwise, $\del_{w^\mo(\be)<0} = \del_{u^\mo(\be)<0}$ because $\be$ (or $s(\be)$) is not bad. \qedhere
    \end{itemize}
\end{proof}

Now we compute the intersection $V(c)\tilw^*(\tau) \cap S_\F^\circ(\tilz')^{\nbl_{\bfa_{\tau'}}}\tilw^*(\tau')$. This is done most naturally in terms of Schubert cells. Recall from the proof of Lemma \ref{lem:V(c)} that
\begin{align*}
    V(c) = u^{\dia*} \Ad_{w_0}( N_{(w_0u^\dia)^*}^{\nblat'}) t_{\eta} w^{-\dia *}.
\end{align*}

\begin{lem}\label{lem:open-locus}
    The intersection $V(c)\tilw^*(\tau) \cap S_\F^\circ(\tilz')^{\nbl_{\bfa_{\tau'}}}\tilw^*(\tau')$ contains the open subscheme of $V(c)\tilw^*(\tau)$ given by matrices of the form $u^{\dia*}  A  t_{\eta}  w^{-\dia *} \tilw^*(\tau)$ where for $\F$-algebra $R$,
    \begin{enumerate}
        \item $A\in \Ad_{w_0}N_{(w_0u^\dia)^*}^{\nblat'}(R)$;
        \item $A_{-\al}\ix{j_0+1} \in R\DB{v}^\times$; 
        \item for each $2 \le i \le k_0-i_0$, the $i\times i$-minor of $A\ix{j_0+1}s_{\al_{i_0k_0}}$ given by rows and columns in $\{ k_0-i+1, \dots, k_0 \}$ is in $R\DB{v}^\times$.
    \end{enumerate}
\end{lem}

\begin{proof}
    Since $(\tilz')^* = u^{-\dia} \tilw_h^\mo w_0 u^\dia$, \eqref{eqn:schubert-modification} shows that
    \begin{align*}
        S_\F^\circ(\tilz')^{\nblat} = (S^\circ_\F((w_0 u^\dia)^*) (u^{-\dia}\tilw_h^\mo)^* )^{\nblat} \subset u^{\dia*} \Ad_{w_0}(N_{(w_0 u^\dia)^*}) t_{\eta} u^{-\dia*}.
    \end{align*}
    We compare this multiplied by $\tilw^*(\tau')$ with the description of $V(c)\tilw^*(\tau)$. After canceling common factors, we get
    \begin{align*}
        u^{\dia*}  \Ad_{w_0}(N_{(w_0 u^\dia)^*}) \pi^\mo(s_\al), \ \ u^{\dia*}  \Ad_{w_0}(N_{(w_0 u^\dia)^*}).
    \end{align*}
    (Note that we ignore the monodromy condition for now. It can be imposed after understanding the intersection.) They are different only at the embedding $(j_0+1)$. Thus, we work with the case $f=1$ and $\cJ=\{*\}$ until the end of the proof. Consider $A\in \Ad_{w_0}(N_{(w_0 u^\dia)^*})$ satisfying the conditions (1--3). Since we are working with subvarieties of $\Fl$, we need to show that
    \begin{align*}
        \cI u^{\dia *} A s_\al \cap \cI u^{\dia*}  \Ad_{w_0}(N_{(w_0 u^\dia)^*})\neq \emptyset. 
    \end{align*}
    Note that
    \begin{align*}
        \cI u^{\dia*}  \Ad_{w_0}(N_{(w_0 u^\dia)^*}) = \cI u^{\dia*} w_0 \cI w_0 \supset \cI u^{\dia*}  L^+\ov{U} 
    \end{align*}
    and $\Ad_{w_0}(N_{(w_0 u^\dia)^*}) \subset L^+\ov{U}$. Therefore, $\cI u^{\dia*}  \Ad_{w_0}(N_{(w_0 u^\dia)^*}) = \cI u^{\dia*}  L^+\ov{U} $ and we only need to show that 
    \begin{align*}
        u^{-\dia*}\cI u^{\dia *} A s_\al \cap   L^+\ov{U}\neq \emptyset.
    \end{align*}
    We interpret the left multiplication by $u^{-\dia*}\cI u^{\dia *}$ as certain row operations. We verify that $A s_\al$ can be transformed into a lower triangular matrix by performing such row operations.

    At the level of $R$-points, $\Ad_{u^{-\dia*}}(\cI)$ is the matrix whose diagonal entries are in $R\DB{v}^\times$ and for $\be=\al_{ik}\in \Phi$, its $\be$-entry is in $v^{\del_{u^\mo(\be)<0} - \RG{\nu_u,\be^\vee}}\F\DB{v}$. Note that the exponent of $v$ is equal to $-m_{u^\dia,\be}$. Since $u^\dia$ is dominant, $-m_{u^\dia,\be}<0$ if $\be\in \Phi^+$. This means that for $f(v)\in \R\DB{v}$, $u_{\be}(f(v))\in \Ad_{u^{-\dia*}}(\cI)(R)$, and the left multiplication  by $u_{\be}(f(v))$ corresponds to the row operation replacing $R_i$ by $R_i+f(v)R_j$. In fact, we will only perform such row operations and the row operation multiplying $f(v)\in R\DB{v}^\times$ to a given row. 

    Since $A s_\al$ is just $A$ with $i_0$th and $k_0$th columns exchanged, the condition $A_{-\al}\in \F\DB{v}^\times$ means that we can perform row operations given by $u_{\be}(\F\DB{v})$ for $\be=\al_{tk_0}$ with $i_0\le t < k_0$ to remove $\be$-entry of $A s_\al$. This will introduce new upper triangular entries at $\be'=\al_{tk}$ for $t<k<k_0$, but they can be removed by performing row operations given by $u_{\be'}(R\DB{v})$, using the fact that $(A s_{\al})_{kk}=1$. At the end, we obtain a lower triangular matrix whose entries are in $R\DB{v}$. It remains to show that the diagonal entries are in $R\DB{v}^\times$, so that we can multiply elements in $R\DB{v}^\times$ on each row to turn the matrix into unipotent lower triangular. This follows from condition (3) noting that the row operations we have performed do not change the relevant minors.
\end{proof} 

\begin{defn}
    Let $\be=\al_{ki}\in \Phi^-$. We define $\fD_{\be}$ to be the set of pairs $(\be_1,\be_2)\in (\Phi^-)^2$ such that $\be_1=\al_{ti}$ and $\be_2=\al_{kt}$ and $P_{\be}$ to be the set of tuples $(\be_1,\dots,\be_s)\in (\Phi^-)^s$ of size $s\ge 1$ such that $\be_{m}=\al_{t_{m}t_{m-1}}$ with $t_{0}=i$ and $t_s = k$. In particular, $\sum_{m=1}^s \be_m = \be$.
\end{defn}

\begin{rmk}\label{rmk:open-conditions}
    We make the conditions (2, 3) in Lemma \ref{lem:open-locus} explicit. Note that these conditions depend only on $A\mod v$. For $\be\in \Phi^-$, we write $a_\be := (A\mod v)_\be$. Then condition (2) says that $a_{-\al}\in R^\times$.

    Now fix $i$ between $2$ and $k_0-i_0$. The $i\times i$-minor modulo $v$ in condition (3) is (up to sign) the determinant of the following matrix $M_i$ obtained by taking rows in $\{ k_0-i+1, \dots, k_0 \}$ and columns in $\{i_0, k_0-i+1, \dots, k_0-1 \}$ of $A \mod v$.
    \begin{align*}
       M_i:= \pma{a_{(k_0-i+1)i_0} & 1 & 0 & \text{\rotatebox[origin=c]{170}{$\cdots$}} & 0 \\
        a_{(k_0-i+2)i_0} & a_{(k_0-i+2)(k_0-i+1)}& 1 & \text{\rotatebox[origin=c]{170}{$\dots$}} & 0 & \\
    \text{\rotatebox[origin=c]{90}{$\cdots$}} & \text{\rotatebox[origin=c]{90}{$\cdots$}} & & 
    & 1 \\
        a_{-\al} & a_{k_0(k_0-i+1)} & \cdots & & a_{k_0(k_0-1)}}
    \end{align*}
    By applying the cofactor expansion with respect to the first row, we have
    \begin{align*}
        \det(M_i) = a_{(k_0-i+1)i_0} \det(M'_{i-1}) - \det(M_{i-1})
    \end{align*}
    where $M'_{i-1}$ is the $(i-1)\times (i-1)$ matrix obtained by removing the first row and column of $M_i$. For $\be:= \al_{k_0(k_0-i+1)}$, one can inductively check that 
    \begin{align*}
       X_{\be}:= \det(M'_{i-1}) = \sum_{(\be'_1,\dots,\be'_s)\in P_{\be}} (-1)^{i-1-s}a_{\be'_1}\dots a_{\be'_s}.
    \end{align*}
    Now using the inductive formula for $\det(M_i)$, we can show that
    \begin{align*}
        (-1)^{i-1}\det(M_i) = a_{-\al} + \sum_{\substack{(\be_1,\be_2)\in \fD_{-\al} \\ \be_1=\al_{(k_0-i'+1)i_0}, i'\le i}} (-1)^{i'-1} a_{\be_1} X_{\be_2}. 
    \end{align*}
    In particular, when $i=k_0-i_0$, we have
    \begin{align*}
        \det(M_{k_0-i_0}) = \sum_{(\be_1,\dots,\be_s)\in P_{-\al}}  (-1)^{k_0-i_0-s}a_{\be_1}\cdots a_{\be_s}.
        \end{align*}
\end{rmk}

We now finish the proof of Lemma \ref{lem:int} by intersecting the open subscheme in Lemma \ref{lem:open-locus} and $V(Z_{-\al})\tilw^*(\tau)$. By the second claim of Lemma \ref{lem:V(c)}, we can view $Z_{-\al}$ as a function on $\Ad_{w_0}(N^{\nblat'}_{(w_0u^\dia)^*})$. Then it remains to show that the vanishing locus of $Z_{-\al}$ is not disjoint from the open locus given by Lemma \ref{lem:open-locus}.

Suppose that $f>1$. Then $Z_{-\al}$ is a coordinate of the $j_0$th matrix of $\Ad_{w_0}(N^{\nblat'}_{(w_0u^\dia)^*})$ while the open conditions are imposed on $(j_0+1)$th matrix. Thus, the closed and open loci are not disjoint.

Suppose that $f=1$. Let $A\in \Ad_{w_0}(N^{\nblat'}_{(w_0u^\dia)^*})$ be the universal matrix and $c_\be$ be the top-degree coefficient of $A_{\be}$ for $\be\in \Phi^-$. By \cite[Proposition 4.3.2]{LLMPQ-CL}, $c_\be$'s are precisely the affine coordinates of $V(c)$. Recall the definition of $Z_{-\al}$ (Lemma 4.1.3 in \loccit):
\begin{align}\label{eqn:Z}
    Z_{-\al} = (m_{-\al} - \RG{\bfa_\tau,-w^\mo(\al)^\vee} )c_{-\al} + \sum_{\substack{(\be_1,\be_2) \in \fD_{-\al}
    \\
    \be_1=\al_{(k_0-i+1)i_0}}} (m_{\be_2} + \kappa_{\be_2} - \RG{\bfa_\tau,w^\mo(\be_2)^\vee}) c_{\be_2}  \sum_{(\be'_1,\dots,\be'_s)\in I_{\be_1}}(-1)^{i-1-s} c_{\be'_1}\cdots c_{\be'_s}.
\end{align}
where $I_{\be_1}$ is the subset of $P_{\be_1}$ whose elements satisfy $\sum_{t=1}^s\del_{w^\mo(\be'_t)>0} = \del_{w^\mo(\be_1)>0}$. (Note that we used $\del_{w^\mo(\be'_t)<0} = 1- \del_{w^\mo(\be'_t)>0}$.)

Suppose that $m_{u^\dia,\al}>0$. This means that $\deg A_{-\al}>0$. In particular, $c_{-\al} \neq a_{-\al}$. By \cite[Proposition 4.3.2]{LLMPQ-CL}, $a_{-\al}$ is determined by $c_{-\al}$ and $c_{\be}$ for $-\al<\be$. By appealing to a more concrete calculation in \cite[Proof of Theorem 4.2.4]{LLLMlocalmodel}, we can see that the expression for $a_{-\al}$ contains the monomial $c_{\al_{k_0(k_0-1)}}c_{\al_{(k_0-1)(k_0-2)}}\dots c_{\al_{(i_0+1)i_0}}$.

In contrast, the monomial $c_{\al_{k_0(k_0-1)}}\dots c_{\al_{(i_0+1)i_0}}$ does not appear in the right hand side of \eqref{eqn:Z} by Lemma \ref{lem:partition}(1). This also means that all monomials appearing there are divisible by $c_{-\be}$ for some $\be \in \Phi^+\backslash\Del$. Thus, the vanishing locus of $c_{-\be}=0$ for all $\be \in \Phi^+\backslash\Del$ is contained in  $V(Z_{-\al})$, and the restriction of $a_{-\al}$ on it is given by a (non-zero) scalar multiple of $c_{\al_{k_0(k_0-1)}}\dots c_{\al_{(i_0+1)i_0}}$. Moreover, for $2\le i < k_0-i_0$, $(-1)^{i-1}\det(M_i) = a_{-\al}$ and $\det({M_{k_0-i_0}}) = C c_{\al_{k_0(k_0-1)}}\dots c_{\al_{(i_0+1)i_0}}$ where $C\neq 0$ by the genericity condition.  
This shows that the triple intersection contains the locus where $c_{-\be}\neq 0$ for $\be\in \Del$ and $c_{-\be}=0$ otherwise.

Finally, we consider the case of $m_{u^\dia,\al}=0$. Then $A_{\be} = c_{\be} = a_{\be}$ for all $-\al \le \be$. Using Lemma \ref{lem:partition}(2) below, we have
\begin{align*}
    (m_{-\al} - \RG{\bfa_\tau,-w^\mo(\al)^\vee} )^\mo Z_{-\al} = a_{-\al} +  \sum_{\substack{(\be_1,\be_2) \in \fD_{-\al}
    \\
    \be_1=\al_{ti_0}}} (-1)^{t+k_0} a_{\be_1} Y_{\be_2} 
\end{align*}
where we define
\begin{align*}
    Y_{\be_2} = \sum_{(\be_1',\dots,\be'_s)\in P_{\be_2}} (-1)^{t+k_0-s} \frac{m_{\be'_s}+\kappa_{\be_s'} - \RG{\bfa_\tau,w^\mo(\be'_s)^\vee}}{m_{-\al} - \RG{\bfa_\tau,-w^\mo(\al)^\vee} } a_{\be_1'}\dots a_{\be'_s}.
\end{align*}
We need to compare the vanishing of $Z_{-\al}$ with non-vanishing of $\det(M_i)$ we computed in Remark \ref{rmk:open-conditions}.

We can think of $Z_{-\al}$ and $\det(M_i)$ for $2 \le i \le k_0-i_0$ as linear functions in variables $a_{\be_1}$ and treat $X_{\be_2}$ and $Y_{\be_2}$ as coefficients where $(\be_1,\be_2)\in \fD_{-\al}$. We can also assume that $X_{\be_2}$ and $Y_{\be_2}$ are invertible (e.g.~by plugging in generic values for $a_{\be'}$ with $\be'\ge \be_2$). It suffices to show that these functions are linearly independent. Using linear combinations between $\det(M_i)$'s, we can replace $\det(M_i)$ for $2<i< k_0-i_0$ by $a_{(k_0-i+1)i_0}X_{k_0(k_0-i+1)}$. When $i=2$, $-\det(M_2) = a_{-\al} + a_{(k_0-1)i_0}X_{k_0(k_0-1)}$.  Similarly, we can also replace $Z_{-\al}$ by $ a_{-\al} + a_{(k_0-1)i_0}Y_{k_0(k_0-1)}$. Then we only need to show that $a_{\al} + a_{(k_0-1)i_0}X_{k_0(k_0-1)}$ and $a_{\al} + a_{(k_0-1)i_0}Y_{k_0(k_0-1)}$ are linearly independent, which follows from $X_{k_0(i_0+1)}\neq Y_{k_0(i_0+1)}$ (this step requires $\bfa_\tau$ to be $n$-generic). This finishes the proof of Lemma \ref{lem:int} and thus that of Proposition \ref{prop:int}. \qed

\begin{lem}\label{lem:partition}
    \begin{enumerate}
        \item Suppose that $m_{u^\dia,\al}>0$. Then for $\be_t':= \al_{(i_0+t)(i_0+t-1)}$, we have
        \begin{align*}
            \sum_{t=1}^{k_0-i_0-1} \del_{w^\mo(\be'_t)>0} > \del_{w^\mo(\al_{(k_0-1)i_0})>0}.
        \end{align*}

        \item Suppose that $m_{u^\dia,\al}=0$. For $(\be_1,\be_2)\in \fD_{-\al}$ and $(\be'_1,\dots, \be'_s)\in I_{\be_1}$, we have
    \begin{align*}
       \sum_{i=1}^s \del_{w^\mo (\be'_i)>0} = \del_{w^\mo(\be_1)>0}.
    \end{align*}    
    \end{enumerate}
\end{lem}
\begin{proof}
    For item (1), $m_{w^\dia,\al}\ge 2$. Note that  $m_{w^\dia,-\be}=0$ for $-\be\in \Del$ because $w^\dia$ is restricted. In particular, 
    \begin{align*}
        \RG{\nu_w,-\be^\vee} = \del_{w^\mo(-\be)<0} = \del_{w^\mo(\be)>0}.
    \end{align*}
    By the optimal superadditivity (Remark \ref{rmk:superadd}), we have $m_{w^\dia,\al_{(k_0-1)i_0}}>0$. Then the claim follows from
    \begin{align*}
        m_{w^\dia,-\al_{(k_0-1)i_0}} &= \RG{\nu_w,-\al_{(k_0-1)i_0}} - \del_{w^\mo(-\al_{(k_0-1)i_0})<0} \\
        &= \sum_{t=1}^{k_0-i_0-1} \RG{\nu_w, -\be_t'^\vee} - \del_{w^\mo(\al_{(k_0-1)i_0})>0} \\
        &= \sum_{t=1}^{k_0-i_0-1} \del_{w^\mo (\be'_t)>0} - \del_{w^\mo(\al_{(k_0-1)i_0})>0} > 0.
    \end{align*}

    For item (2), note that $m_{u^\dia,\al}=0$ implies $m_{u^\dia,-\be_1} = 0$ and $m_{u^\dia,-\be_1}=m_{w^\dia,-\be_1}$. In particular $m_{w^\dia,-\be'_t}=0$. Then the claim follows from a similar argument as in item (1).
\end{proof}

\subsection{Supersingular Galois representations}\label{sub:satake}
We start by recalling the following result.

\begin{thm}[{\cite[Theorem 5.3.4]{lee-satake}}]
    Let $\sig$ be a $2n$-deep Serre weight. Then there is a natural isomorphism
    \begin{align*}
        \ov{\Psi}_\sig : \cH(\sig)\risom \cO(\cC_{\sig}).
    \end{align*}
\end{thm}

\begin{defn}\label{def:ss-rhobar}
    Let $\sig$ be a $2n$-deep Serre weight. 
    \begin{enumerate}
        \item For $\rhobar\in \cC_{\sig}(\F)$, we define $\chi_{\sig}(\rhobar)$ to be the character given by the composition 
        \begin{align*}
            \cH(\sig)\xra{\ov{\Psi}_\sig} \cO(\cC_{\sig}) \xra{\ev_{\rhobar}} \F.
        \end{align*}
        \item For $i=1,\dots, n$, we define $f_i := \ov{\Psi}_\sig(T_n)$. We say that $\rhobar \in \cC_{\sig}(\F)$ is \textit{ordinary} (resp.~\textit{supersingular}) \textit{with respect to $\sig$} if $f_i(\rhobar) \neq 0$ (resp.~$f_i(\rhobar)=0$) for all $i=1,\dots, n-1$. We define $\cC_\sig^\ord$ (resp.~$\cC_\sig^\ss$) to be the non-vanishing (resp.~vanishing) locus of $f_i$ for all $i=1,\dots, n-1$.
    \end{enumerate}
\end{defn}

\begin{rmk}\label{rmk:functions}
        Let $\sig$ and $\tau$ be as in Setup \ref{setup}. The functions $f_i\in \cO(\cC_{\sig})$ can be defined as restrictions of functions on $\cX^{\eta,\tau}_{n}$ computing products of normalized Frobenius eigenvalues of Weil--Deligne representations attached to $\rho \in \cX^{\eta,\tau}_{n}(\cO)$. More precisely, consider $\rho \in \cX^{\eta,\tau}_n(\tilz)(\cO)$ for some $\tilz \in \Adm^\vee(\eta)$ with a preimage $(A_j)_\jj \in \tld{U}^{\eta,\nblat}(\tilz)$. Let $\ov{A}_j$ be the matrix obtained by taking diagonal entries of $A_j\mod v$. Then $f_i(\rhobar)$ is equal to the mod $p$ reduction of the top left $i\times i$ minor of $(\prod_{j=f-1}^0 w_j \ov{A}_j w_j^\mo)^\mo$ multiplied by $p^{fi(2n-i-1)/2}$. See \cite[\S4.2]{lee-satake} for details.
\end{rmk}

The following result justifies the usefulness of the triple intersection in Proposition \ref{prop:int}.

\begin{prop}\label{prop:ord-ss}
    We follow Setup \ref{setup}.
    \begin{enumerate}
        \item We have $\cC^{\tau'}_{\sig'} \subset \cC_{\sig'}^{\ord}$.
        \item Suppose that $w^\dia$ is special. Then $\cC_\sig^\tau(\tilz)\cap \cC^{\tau}_{\sig'}(\tilz) \cap \cC_{\sig'}^{\tau'}(\tilz') \subset \cC_\sig^\ss$.
    \end{enumerate}
\end{prop}

\begin{proof}
    Recall that $\sig(\tau')\simeq \Ind_{B(k)}^{G(k)}[\tld{\lam}']$ where $\tld{\lam}'=(\sig')^{U(k)}$. For $\rhobar \in \cC_{\sig'}^{\tau'}(\tilz')(\F)$, we can compute $f_i(\rhobar)$ by finding a lift $\rho\in \cX^{\eta,\tau'}_{n}(\tilz')(\cO)$ and following Remark \ref{rmk:functions}. Then item (1) follows from the explicit local chart Proposition \ref{prop:chart}(1).

    For item (2), we repeat the above process using $\tau$ instead of $\tau'$ and Proposition \ref{prop:chart}(2). Note that the specialness condition means $(i_0,k_0)=(1,n)$. Since $\rhobar\in \cC_\sig^\tau(\tilz)\cap \cC^{\tau}_{\sig'}(\tilz)(\F)$, we can choose $\rho$ lifting $\rhobar$ with preimage $(A_j)_{\jj}\in \tld{U}^{\eta,\nblat}(\tilz)$. Using \cite[Proposition 3.1.2]{LLMPQ-CL}, we have
\begin{align*}
w_{j_0}A\ix{j_0}w^\mo_{j_0} = u_{-\al}(c)s_{\al}(v+p)^\eta V
\end{align*}
where $V$ is lower-triangular unipotent with entries in $\cO[v]$. Note that the function $c$ denotes $n$th diagonal entry of $w_{j_0}A\ix{j_0}w^\mo_{j_0}$. For $2\le m \le n-1$, its $m$th diagonal entry is of the form $c_m(v+p)^{n-m}$ where $c_m\in \cO^\times$, and we denote by $d$ the constant term of its first diagonal entry. Since $f_n$ is a unit, the computation of $f_n(\rhobar)$ as explained in Remark \ref{rmk:functions} shows that $cd\in p^{n-1}\cO^\times$.  Since $\rhobar$ is in the vanishing locus of $c$, this means that the $p$-adic valuation of $d$ is less than $n-1$. Again using Remark \ref{rmk:functions}, this implies that $f_i(\rhobar)=0$ for $i=1,\dots,n-1$.
\end{proof}


\begin{cor}\label{cor:ss}
    Let $\sig=F_{(w^\dia,\om)}$ be a $(3n-4)$-deep Serre weight. If $w^\dia$ is special, then $\cC_{\sig}^\ss$ is of codimension one, i.e.~$\dim \cC_{\sig}^\ss = \dim \cC_\sig -1$. 
\end{cor}

\begin{proof}
    We follow the notation of Setup \ref{setup}. It suffices to show that the triple intersection in Proposition \ref{prop:ord-ss}(2) is of codimension one. This is because at the level of local models, the non-empty triple intersection in Lemma \ref{lem:int} is the vanishing locus of a function $Z_{-\al}$ on $V(c)\tilw^*(\tau) \cap S_\F^\circ(\tilz')^{\nbl_{\bfa_{\tau'}}}\tilw^*(\tau')$, and the latter is of dimension $n(n-1)f/2$. 
\end{proof}

\begin{rmk}\label{rmk:inftymany}
   \begin{enumerate}
       \item 
       The supersingular locus $\cC_{\sig}^\ss$ is the vanishing locus of $n-1$ functions. Thus, Corollary \ref{cor:ss} shows that $f_1,\dots, f_{n-1} \in \cO(\cC_\sig)$ are far from forming a regular sequence when $w^\dia$ is special. Note that they do form a regular sequence if $w^\dia \in \Omega^\cJ$ (this can be checked using \cite{LLMPQ-FL}).
       \item 
       It is clear from the proof of Proposition \ref{prop:int} that given $t\in \Fpbar^\times$, there are infinitely many $\rhobar_0 \in (\cC_\sig^\tau(\tilz)\cap \cC^{\tau}_{\sig'}(\tilz) \cap \cC_{\sig'}^{\tau'}(\tilz'))(\Fpbar)$ such that $\chi_{\sig'}(\rhobar_0)$ are all distinct and $\chi_{\sig'}(\rhobar_0)(T_n)=t$.
   \end{enumerate} 
\end{rmk}

\section{Global methods}

\subsection{Global setup}\label{sub:global}
In this subsection, we briefly explain a candidate for the mod $p$ local Langlands correspondence given by a global method. We refer \cite[\S10.3]{LLMPQ-FL} for more detail.

Let $\rhobar_0:G_K \ra \GL_n(\F)$ be a continuous representation. We assume that $p\nmid 2n$ and $\rhobar_0$ admits a potentially diagonalizable potentially crystalline lift of regular weight. By \cite[Corollary A.7]{EG-geom_BM-MR3134019}, we can find a \textit{suitable globalization} of $\rhobar_0$ which is a triple $(F,F^+,\rbar)$ of a CM field $F$, its maximal totally real subfield $F^+$, and a continuous representation $\rbar: G_{F^+} \ra \tld{G}(\F)$. Here, $\tld{G}$ is the Clozel--Harris--Taylor group scheme which is a certain semi-direct product of $\GL_n\times \GL_1$ and $\Z/2\Z$. We refer \cite[\S2.1]{6author} for the precise definition. The important point for us is that each place $v \mid p$ in $F^+$ splits in $F$ and $F^+_v \simeq K$ and after choosing such isomorphisms, $\rbar|_{G_{F^+_v}}$ maps into $\GL_n\times \GL_1$ and is isomorphic to $\rhobar_0\times \det(\rhobar_0)$. Moreover, $\rbar$ is potentially diagonalizably automorphic and unramified at places $v \nmid p$.

By solvable base change, we can and do assume that $[F^+:\Q]$ is even. We let $\cG_{/F^+}$ be a reductive group which is an outer form of $\GL_n$ splitting over $F$ such that $\cG(F_{v}^+)\simeq U_n(\R)$ for all $v\mid \infty$. By abuse of notation, we also denote the reductive model of $\cG$ over $\cO_{F^+}[1/N]$ for some $N\in \Z$ coprime to $p$ again by $\cG$. 

From now on, we fix a place $\tld{v} \mid p$ in $F$ and write $v:= \tld{v}|_{F^+}$. Let $S(U, \F)$ be the space of algebraic automorphic forms on $\cG$ of level $U = U^vU_v$ and coefficients $\F$. We have a set $\cP_U$ of finite places $\tld{w}$ of $F$ such that $\tld{w}|_{F^+}$ splits in $F$, $w \nmid pN$, and $U$ is unramified at $\tld{w}|_{F^+}$. For a subset $\cP \subset \cP_U$ of finite complement, let $\T^\cP$ be the abstract Hecke algebra generated by the usual double coset operators at places in $\cP$ with a maximal ideal $\fm_{\rbar}$ determined by $\rbar$. Then fixing a tame level $U^v \subset \cG(\A^{\infty,v}_{F^+})$, we define
\begin{align*}
    \pi(\rbar):= \varinjlim_{U_v\le \cG(\cO_{F_v^+})} S(U^vU_v, \F)[\fm_{\rbar}].
\end{align*}
It is a major open conjecture that $\pi(\rbar)$ is \textit{purely local}, i.e.~that it depends only on $\rhobar_0$.

\begin{rmk}
    We follow Setup \ref{setup}. For application, we apply the above formalism to $\rhobar_0$ in the triple intersection in Proposition \ref{prop:int}. The potentially crystalline deformation ring $R^{\eta,\tau'}_{\rhobar_0}$ is formally smooth (this follows from Proposition \ref{prop:chart}). In particular, potentially crystalline lifts of $\rhobar_0$ of type $(\eta,\tau')$ exist and are potentially diagonalizable.  
\end{rmk}

\subsection{Serre weights of Galois representations}
We keep the notation from the previous subsection.
\begin{defn}
    We define a set of Serre weights of $G(k)$
    \begin{align*}
        W_v(\rbar) := \{ \sig \mid \Hom_{\rmK}(\sig, \pi(\rbar)|_{\rmK}) \neq 0\}.
    \end{align*}
\end{defn}
Serre weight conjectures usually concern a variant of $W_v(\rbar)$ by considering all places $v$ dividing $p$. Since we are interested in purely local applications, $W_v(\rbar)$ is more suitable for us.

It is expected that $W_v(\rbar)$ depends only on $\rhobar_0$. If $\rhobar_0$ is tamely ramified and sufficiently generic, an explicit (conjectural) description of $W_v(\rbar)$ in terms of $\rhobar_0|_{I_K}$ was proposed by Herzig \cite{HerzigDuke}. For general $\rhobar_0$, it is expected that $W_v(\rbar)$ can be described in terms of conjectural Breuil--M\'ezard cycles (see \cite[\S8.4]{EGstack} and \cite[Conjecture 8.1.1]{LLLMlocalmodel}). In generic cases, Breuil--M\'ezard cycles are recently constructed in \cite{FLH} for unramified groups. Using the expected properties of Breuil--M\'ezard cycles (\cite[Remark 8.1.2]{LLLMlocalmodel}), we have the following weak version of Serre weight conjectures.

\begin{conj}
    Let $W^g(\rhobar_0)$ be the set of Serre weights $\sig$ of $G(k)$ such that $\rhobar_0\in \cC_\sig(\F)$. Then
    \begin{align*}
        W^g(\rhobar_0) \subset W_v(\rbar).
    \end{align*}
\end{conj}

We need an even weaker version of the above conjecture.

\begin{prop}\label{prop:SWC}
    We follow Setup \ref{setup}. In addition, we assume that $\sig$ is $3n$-generic. Suppose that $\rhobar_0$ is in the triple intersection in Proposition \ref{prop:int}. Then
    \begin{align*}
         W^g(\rhobar_0) \cap \JH(\osig(\tau)) = W_v(\rbar) \cap \JH(\osig(\tau)).
    \end{align*}
    Note that $W^g(\rhobar_0) \cap \JH(\osig(\tau)) = \{\sig, \sig'\}$. Furthermore, $\sig'$ is a Jordan--H\"older factor of $\osig(\tau)$ with multiplicity one.
\end{prop}

\begin{proof}
    This follows from the standard Taylor--Wiles patching argument. We give a brief sketch. Let $M_\infty$ be a weak minimal potentially diagonalizable patching functor for $\rhobar_0$ \cite[Proposition 6.2.4(2)]{LLLMlocalmodel}. 
    
    Note that $\rhobar_0^\ss|_{I_K} = \rhobar'$. Since $\sig$ is $3n$-generic, $\rhobar'$ is $(2n+1)$-generic. Then \cite[Theorem 6.1]{LLLM-DL} shows that $W_v(\rbar)\subset W^?(\rhobar')$. Here, $W^?(\rhobar')$ is the explicit set of Serre weights attached to (tamely ramified) $\rhobar'$ by Gee--Herzig--Savitt \cite{GHS-JEMS-2018MR3871496}.  This implies that for a Serre weight $V$, $M_\infty(V)\neq 0$ only if $V\in W^?(\rhobar')$. Since $\rhobar_0\in \cC_{\sig'}^{\tau'}(\tilz')$ and $\tilz'$ is extremal, the potentially crystalline deformation ring $R^{\eta,\tau'}_{\rhobar_0}$ is formally smooth. This implies that potentially crystalline lifts of $\rhobar_0$ of type $(\eta,\tau')$ are potentially diagonalizable. Then $M_\infty(\osig(\tau'))\neq 0$. Then $\sig' \in W_v(\rbar)$ follows from $\JH(\sig(\tau'))\cap W^?(\rhobar')=\{\sig'\}$. 
    
    To show that $\sig \in W_v(\rbar)$, we choose a tame inertial type $\tau''$ such that $\tilw(\rhobar',\tau'')=w^\mo t_\eta u$. Then $\JH(\osig(\tau''))\cap W^?(\rhobar') = \{\sig,\sig'\}$, and the deformation ring $R^{\eta,\tau''}_{\rhobar_0}$ is a power series ring over $\cO\DB{x,y}/(xy-p)$ by \cite[Theorem 5.2.3]{LLMPQ-CL}. Since $R^{\eta,\tau''}_{\rhobar_0}$ is normal and $M_\infty(\sig)$ is non-zero,  $M_\infty(\sig^\circ(\tau''))$ is a non-zero module supported on $\Spec R^{\eta,\tau''}_{\rhobar_0}$ (it is a module over $R_\infty^{\eta,\tau''}$ in \loccit, but we can simply view it as $R^{\eta,\tau''}_{\rhobar_0}$-module as $R_\infty^{\eta,\tau''}$ is an $R^{\eta,\tau''}_{\rhobar_0}$-algebra). Since $M_\infty(\osig(\tau''))$ is an extension of $M_\infty(\sig)$ and $M_\infty(\sig')$, and the support of $M_\infty(\osig(\tau''))$ (which is $\Spec R^{\eta,\tau''}_{\rhobar_0}/\varpi$) is strictly larger than the support of $M_\infty(\sig)$ (which is irreducible), $M_\infty(\sig')$ is also non-zero. In other words, $\sig \in W_v(\rbar)$. This also shows that  $W_v(\rbar) \cap \JH(\osig(\tau'')) = \{\sig, \sig'\}$, which implies $\{\sig, \sig'\} \subset W_v(\rbar) \cap \JH(\osig(\tau))$. 
    
    The claimed equality is not immediate because $\JH(\osig(\tau))\cap W^?(\rhobar')$ strictly contains $\{\sig,\sig'\}$. Still, $M_\infty(\sig^\circ(\tau''))$ and $M_\infty(\sig^\circ(\tau))$ are generically free of rank one over abstractly isomorphic base rings $R^{\eta,\tau''}_{\rhobar_0}$ and $R^{\eta,\tau}_{\rhobar_0}$, respectively. By \cite[Lemma 2.2.10]{EG-geom_BM-MR3134019}, the cycles attached to $M_\infty(\osig(\tau''))$ and $M_\infty(\osig(\tau))$ are given by the cycle attached to the special fiber of their base rings. In particular, they can be identified under the isomorphism between base rings. (In fact, the special fiber of the two base rings are \textit{identical} as quotients of $R_{\rhobar_0}^{\square}$. We do not need this here.) This proves that $\{\sig, \sig'\} =W_v(\rbar) \cap \JH(\osig(\tau))$. Finally, $\sig'$ is a Jordan--H\"older factor of $\osig(\tau)$ with multiplicity one because $R^{\eta,\tau}_{\rhobar_0}/\varpi$ is reduced and thus its cycle is multiplicity free.
\end{proof} 

\begin{cor}\label{cor:factorization}
    We keep the notation and assumptions in the previous Proposition. We assume that $w^\dia$ in Setup \ref{setup} is special. Then there is a non-zero morphism $\cind_\KZ^\rmG \sig \ra \pi(\rbar)$ (resp.~$\cind_\KZ^\rmG \sig' \ra \pi(\rbar)$). Moreover, any such non-zero morphism factors through $\SS(\sig)$ (resp.~$\cind_\KZ^\rmG \sig'\otimes_{\cH(\sig')}\chi_{\sig'}(\rhobar_0)$).
\end{cor}
\begin{proof}
    The first claim follows from Proposition \ref{prop:SWC} and Frobenius reciprocity. 
    The second claim follows from the mod $p$ local-global compatibility result for $\ov{\Psi}_{\sig}$ (and $\ov{\Psi}_{\sig'}$) \cite[Theorem 7.1.1]{lee-satake}. 
\end{proof}

\begin{rmk}\label{rmk:PS}
    The representation $\cind_\KZ^\rmG \sig'\otimes_{\cH(\sig')}\chi_{\sig'}(\rhobar_0)$ in Corollary \ref{cor:factorization} is an irreducible principal series. More precisely, 
    \begin{align*}
        \cind_\KZ^\rmG \sig'\otimes_{\cH(\sig')}\chi_{\sig'}(\rhobar_0) \simeq \Ind_{\ov{B}(K)}^{\rmG} \chi
    \end{align*}
    by \cite[Theorem 3.1]{HerzigDuke}. Here, $\chi=\otimes_{i=1}^n\chi_i$ is a $T(K)$-character over $\F$ such that $\chi|_{T(\cO_K)}$ is equal to the highest weight of $\sig'$ and $\chi_i(p) = \chi_{\sig'}(\rhobar_0)(T_i) / \chi_{\sig'}(\rhobar_0)(T_{i-1})$ for $i=1,2,\dots,n$ (interpret $T_0=1$). The representation $\Ind_{\ov{B}(K)}^{\rmG} \chi$ is irreducible by \cite{Oll-PS}. Distinct choices of $\chi_{\sig'}(\rhobar_0)$ give non-isomorphic principal series (see \cite[Theorem 1.1]{HerzigDuke}).
\end{rmk}

\section{The main result}
Let $\sig$ be a Serre weight of $\GL_n(k)$. For $t\in \F^\times$, we define the universal supersingular representation (with fixed central character determined by $t$)
\begin{align*}
    \SS_{t}(\sig):= \cind_\rmK^\rmG \sig/(T_1,\dots, T_{n-1}, T_n-t).
\end{align*}

\begin{thm}\label{thm:main}
    Let $\sig$ be a Serre weight with a $3n$-deep lowest alcove presentation $(w^\dia,\om)$. Suppose that $w^\dia$ is special. Then $\SS_{t}(\sig)$ is non-admissible and of infinite length for any $t\in \F^\times$.
\end{thm}

\begin{proof}
    We follow Setup \ref{setup} and choose $\rhobar_0$ in the triple intersection in Proposition \ref{prop:int}. Choose a suitable globalization $(F,F^+,\rbar)$ as in \S\ref{sub:global}.
    
    Fix $1 \le i \le n-1$. Then we have a commutative diagram
    \[
    \begin{tikzcd}
        \cind_{\rmK}^{\rmG} C_{i}(\sig) \arrow[r, two heads, "\cyc_i"] \arrow[rd, dashed] & \SS_t(\sig) \arrow[d] \\
        & \pi(\rbar).
    \end{tikzcd}
    \]
    where the horizontal surjection is induced by Proposition \ref{prop:WC} and the (non-zero) vertical arrow is induced by Corollary \ref{cor:factorization}. The diagonal arrow induces a morphism
    \begin{align*}
       f: C_i(\sig) \ra \pi(\rbar)|_{\rmK}
    \end{align*}
    which is non-zero because $\cyc_i$ is surjective. Recall from Lemma \ref{lem:upperbound} that $\sig\notin \JH(C_i(\sig))$ and $\JH(C_i(\sig))\subset \JH(\osig(\tau))$. Then Proposition \ref{prop:SWC} implies that the socle of the image of $f$ isomorphic to $\sig'$. This in particular implies that $\sig'$ is a Jordan--H\"older factor of $C_i(\sig)$ with multiplicity one.

    Let $V\subset C_i(\sig)$ be the smallest subrepresentation such that $\sig'\subset C_i(\sig)/V$. Let $\pi$ be the image of $\cind_\rmK^\rmG \sig'$ in $\SS_t(\sig)/\cyc_i(\cind_{\rmK}^{\rmG} V)$. Note that $\pi$ is a subquotient of $\SS_t(\sig)$ independent of the choice of $\rhobar_0$ and $\rbar$, and the map $\cind_\rmK^\rmG \sig' \ra \pi(\rbar)$ factors through $\pi$. We claim that $\pi$ is non-admissible. This proves the theorem by Lemma \ref{lem:subquot}.
    
    By Corollary \ref{cor:factorization}, the map $\cind_{\rmK}^\rmG \sig' \ra \pi(\rbar)$ induced by $f$ factors through the irreducible quotient $\cind_{\rmK}^\rmG \sig' \otimes_{\cH(\sig')}\chi_{\sig'}(\rhobar_0)$. Thus, $\pi$ admits $\cind_{\rmK}^\rmG \sig' \otimes_{\cH(\sig')}\chi_{\sig'}(\rhobar_0)$ as a quotient. Since there are infinitely many choices of $\rhobar_0$ with fixed determinant and distinct $\chi_{\sig'}(\rhobar_0)$ as in Remark \ref{rmk:inftymany}, $\cind_{\rmK}^\rmG \sig' \otimes_{\cH(\sig')}\chi_{\sig'}(\rhobar_0)$ gives rise to infinitely many non-isomorphic principal series (see Remark \ref{rmk:PS}). 
    This shows that $\pi$ is of infinite length and non-admissible by Lemma \ref{lem:non-adm}.
    \end{proof}

\begin{rmk}\label{rmk:non-special}
    We explain why Theorem \ref{thm:main} for non-special $w^\dia$ is more difficult. Let $n=3$ and $f=1$. In this case, we have two $p$-restricted alcoves, denoted by $C_0$ (the base $p$-alcove) and $C_1$, and only $C_1$ is special.

    Let $\sig=F(\lam)$ with $\lam=(a,b,c)\in C_0$ sufficiently deep. In this case, we have an explicit description of $C_i(\sig)$. For $i=1$, $C_i(\sig)$ is a non-split extension of two Serre weights with socle $\sig':=F(a,c,b-p+1)$ and cosocle $\sig'':=F(c+p-1,a,b)$. For $i=2$, it has the same form with socle $\sig':=F(b+p-1,a,c)$ and cosocle $\sig'':= F(b,c,a-p+1)$. In each case, we can check that $\cC_{\sig} \cap \cC_{\sig''}$ is empty, and although $\cC_{\sig} \cap \cC_{\sig'}$ is non-empty, the intersection $\cC_\sig^{\ss}\cap (\cC_{\sig'} \backslash \cC_{\sig'}^\ss)$ is empty. This is an obstruction to applying our weight cycling argument to the non-admissibility of $\SS_t(\sig)$.

    Roughly speaking, it is expected that $\cC_\sig^\ss$ is more complicated for higher (with respect to $\uparrow$-ordering) $\sig$. When the highest weight of $\sig$ belongs to the base $p$-alcove, $\cC_{\sig}^\ss$ is too simple to contain suitable Galois representations for our weight cycling argument. If the highest weight of $\sig$ does not belong to the base $p$-alcove, we expect that $\cC_{\sig}^\ss$ always contains Galois representations suitable for our argument, but finding such representations would require studying more complicated neighborhoods of potentially crystalline stacks or finding a better description of $C_i(\sig)$.
\end{rmk}

\begin{rmk}\label{rmk:GL2}
    We briefly explain how to prove the non-admissibility of $\SS_t(\sig)$ for $n=2$ and $f\ge 2$ using our strategy. For simplicity, let $f=2$ and $\sig=F(\lam)$ be a sufficiently deep Serre weight. We write $W=\{e,s\}$ and $s_j = \sig_j(s)\in W^{\cJ}$ for $j=0,1$. For $\tau=\tau(e,\lam)$ and $\sig(\tau) = \Ind_{B(k)}^{G(k)}\sig^{U(k)}$, we have
    \begin{align*}
        \JH(\osig(\tau)) = \{F(\lam), F((s_0\cdot \lam)^\dia), F((s_1\cdot \lam)^\dia), F((s_0s_1\lam)^\dia)\}
    \end{align*}
    and $C_1(\sig)$ has three irreducible constituents with socle $\oplus_{j=0,1} F((s_j\cdot \lam)^\dia)$ and cosocle $F((s_0s_1\lam)^\dia)$. Here, for $\mu\in X^*(T)^\cJ$ sufficiently deep in its alcove, $\mu^\dia$ denotes the unique $p$-restricted character (up to $(p-\pi)X^0(T)^\cJ$) in $\mu+(p-\pi)X^*(T)^\cJ$.  More concretely, if we write $\lam=(\lam_0, \lam_1)$ with $\lam_i=(a_i,b_i)\in X_1^*(T)$, then
    \begin{itemize}
        \item $s_0\cdot \lam = ((b_0-1, a_0+1),(a_1,b_1))$ and $(s_0\cdot \lam)^\dia = ((b_0-1+p, a_0+1),(a_1-1,b_1))$;
        \item $s_1\cdot \lam = ((a_0, b_0),(b_1-1,a_1+1))$ and $(s_1\cdot \lam)^\dia = ((a_0-1, b_0),(b_1-1+p,a_1+1))$;
        \item $s_0s_1\lam = ((b_0, a_0),(b_1,a_1))$ and $(s_0s_1\lam)^\dia = ((b_0+p-1, a_0),(b_1+p-1,a_1))$.
    \end{itemize}

    For most Serre weights $\sig'$ of $\GL_2(k)$, it is known that $\cC_{\sig'}$ is smoothly equivalent to $(\P^1)^2$, and the ordinary locus $\cC_{\sig'}^{\ord}$ can be identified with $(\A^1)^2\subset (\P^1)^2$ (see \cite{GKKSW}). We use this to visualize $\cX^{\eta,\tau}_{2,\F}$ as in Figure \ref{fig:stack}.  For each $\sig'\in \JH(\osig(\tau))$, $\cC_{\sig'}$ is depicted as a diamond shape labeled by $\sig'$ in the middle. Its edges (resp.~vertices) should be interpreted as copies of $\P^1$ (resp.~$0$ and $\infty$ in $\P^1$). The gray area corresponds to $\cC_{\sig'}^\ord$. The supersingular locus $\cC_{\sig}^\ss$ is (smoothly equivalent to) two copies of $\P^1$ intersecting at one point, identified with $0\in \P^1$. This is illustrated by the two red edges. Its intersections with $\cC_{\sig'}^{\ord}$ for $\sig' = F((s_j\cdot \lam)^\dia)$ for $j=0,1$ correspond to two one-dimensional families (smoothly equivalent to $\A^1 = \P^1 \backslash \{\infty\}$) of Galois representations of the form 
    \begin{align*}
        \pma{\chi_0 & * \\ 0  & \chi_1} \ \ \text{ and} \ \ \  \pma{\chi_1 & * \\ 0 & \chi_0}
    \end{align*}
    respectively, where $\chi_0|_{I_K} = \oom_{K,\sig_0}^{-(a_0+1) - pb_1}$ and $\chi_1|_{I_K} = \oom_{K,\sig_0}^{-b_0 - p(a_1+1)}$. If we take a non-split $\rhobar_0$ in the first family and run the argument as in the proof of Theorem \ref{thm:main}, the map $C_1(\sig) \ra \pi(\rbar)|_{\rmK}$ has the kernel given by $F((s_1\cdot \lam)^\dia)$, and the image of $\cind_\rmK^\rmG C_1(\sig)$ in $\pi(\rbar)$ is an extension of a supersingular representation of weight $F((s_0s_1\lam)^\dia)$ by a principal series of weight $F((s_0\cdot \lam)^\dia)$ (cf.~Example 1 in \cite[\S2.4.3]{BHHMS2}). For $\rhobar_0$ in the second family, we have the same result but with $F((s_0\cdot \lam)^\dia)$ and $F((s_1\cdot \lam)^\dia)$ interchanged. This argument also shows that the map $C_1(\sig) \ra \cind_\rmK^\rmG \sig/(T_1)|_\rmK$ induced by $\cyc_1$ is injective.
    \begin{figure}
        \centering
        \begin{tikzpicture}[scale=0.9]

\coordinate (L)  at (-5,0);
\coordinate (T)  at (0,2);
\coordinate (R)  at (5,0);
\coordinate (B)  at (0,-2);

\coordinate (LT) at (-2.5,1);
\coordinate (TR) at ( 2.5,1);
\coordinate (RB) at ( 2.5,-1);
\coordinate (BL) at (-2.5,-1);

\coordinate (O) at (0,0);

\def\s{0.9}

\fill[gray!55]
  (L)
  -- ($(L)!\s!(LT)$)
  -- ($(L)!\s!(O)$)
  -- ($(L)!\s!(BL)$)
  -- cycle;

\fill[gray!55]
  (O)
  -- ($(O)!\s!(LT)$)
  -- ($(O)!\s!(T)$)
  -- ($(O)!\s!(TR)$)
  -- cycle;

\fill[gray!55]
  (R)
  -- ($(R)!\s!(TR)$)
  -- ($(R)!\s!(O)$)
  -- ($(R)!\s!(RB)$)
  -- cycle;

\fill[gray!55]
  (O)
  -- ($(O)!\s!(BL)$)
  -- ($(O)!\s!(B)$)
  -- ($(O)!\s!(RB)$)
  -- cycle;

\draw[thick] (L) -- (T) -- (R) -- (B) -- cycle;

\draw[thick] (LT) -- (RB);
\draw[thick] (BL) -- (TR);

\draw[red, very thick] (LT) -- (O) -- (BL);

\node[fill=white, inner sep=1pt] at (-2.5,0) {$F(\lam)$};
\node[fill=white, inner sep=1pt] at (0,1) {$F((s_0\cdot \lam)^{\diamond})$};
\node[fill=white, inner sep=1pt] at (2.5,0) {$F((s_0s_1 \lam)^{\diamond})$};
\node[fill=white, inner sep=1pt] at (0,-1) {$F((s_1\cdot \lam)^{\diamond})$};

\end{tikzpicture}
        \caption{Visualization of $\cX^{\eta,\tau}_{2,\F}$}
        \label{fig:stack}
    \end{figure}
\end{rmk}

\bibliographystyle{amsalpha}
\bibliography{mybib}
\end{document}